\documentclass{article}
\usepackage[utf8x]{inputenc}
\usepackage{amsthm}
\usepackage{float}
\usepackage{mathtools,nccmath}
\usepackage{fullpage}
\usepackage{amssymb}
\usepackage{amsmath}
\usepackage{amsfonts}
\usepackage{latexsym}
\usepackage{enumitem}
\usepackage{url}
\usepackage{bbold}
\usepackage{xcolor}
\usepackage{tkz-tab}
\usepackage[linesnumbered,ruled,vlined]{algorithm2e}
\usepackage{algorithmic}
\usepackage{multirow}

\usetikzlibrary{arrows.meta}

\newcommand\ForAuthors[1]
{\par\smallskip                     
	\begin{center}
		\fbox
		{\parbox{0.9\linewidth}
			{\raggedright\sc--- #1}
		}
	\end{center}
	\par\smallskip                     %
}
\newcommand\comment[1]{}

\def\valu#1{\mathfrak{s}(#1)}
\def\ls#1{\boldsymbol{\ell}(#1)}

\def\funset{\Gamma}
\def\fmon{\funset_0^{\rm dec}}
\def\fmong{\funset_{{\rm ev}}^{\rm dec}}

\def\fmonfung{{\funset}_{\rm f,{\rm ev}}^{\rm dec}}

\def\fset{\Omega([0,1])}
\def\fcste{\funset_{{\rm ev}}^{\rm cst}}

\def\mon{\mathfrak{m}}
\def\cst{\mathfrak{c}}
\def\norm#1{\left\|#1\right\|}
\def\secfunh{\funset_{\rm f}}
\def\secfunb{\funset_{\rm s}}

\def\Fu{\mathfrak{F}_u}
\def\Fgen{\mathfrak{F}_{\cdot}}

\def\Idd{\mathrm{Id}}

\def\rr{\mathbb R}
\def\rd{\rr^d}
\def\nn{\mathbb N}
\def\xin{X^\mathrm{in}}

\def\ff{F}

\def\agmx{\operatorname{Argmax}}

\def\vmax{\operatorname{vmax}}
\def\kmax{\operatorname{kmax}}

\def\gs{\Delta^s}
\def\bigkse{\mathbf{K}}
\def\bigks{\mathbf{K}^s}

\def\gls#1{{\mathrm R}_{#1}}
\def\gsls#1{\gls{#1}^s}
\def\res{\mathcal{S}}

\def\lmax#1{\lambda_{\rm max}(#1)}
\def\lmin#1{\lambda_{\rm min}(#1)}

\newtheorem{prop}{Proposition}
\newtheorem{defi}{Definition}
\newtheorem{example}{Example}
\newtheorem{lemma}{Lemma}
\newtheorem{theorem}{Theorem}
\newtheorem{coro}{Corollary}
\newtheorem{remark}{Remark}
\newtheorem{prob}{Problem}

\SetKwInOut{Input}{Input}
\SetKwInOut{Output}{Output}
\SetKwInOut{init}{Initialisation}

\def\red#1{#1}
\def\blue#1{#1}

\title{On the Maximization of Real Sequences}

\author{Assalé Adjé\footnote{mail: assale.adje@univ-perp.fr}\\ LAboratoire de Modélisation et Pluridisciplinaire et Simulations\\
LAMPS\\ Université de Perpignan Via Domitia\\ France
}

\date{}

\begin{document}
	\maketitle
	
	\begin{abstract}
		In this paper, we study a maximization problem on real sequences. More precisely, for a given sequence, we are interested in computing the supremum of the sequence and an index for which the associated term is maximal. We propose a general methodology to solve this maximization problem. The method is based on upper approximations constructed from pairs of eventually decreasing sequences of strictly increasing continuous functions on $[0,1]$ and scalars in $(0,1)$. Then, we can associate integers with these pairs using the inverse of the functions on $[0,1]$. We prove that such pairs always exist, and one provides the index maximizer. In general, such pairs provide an upper bound for the greatest maximizer of the sequence. Finally, we apply the methodology to concrete examples, including famous sequences such as the logistic, Fibonacci, and Syracuse sequences. We also apply our techniques to norm-based peak computation problems on discrete-time linear systems.   
	\end{abstract}
	
	
	\section{Introduction}
	In this paper, denoting the set of real sequences by $\rr^\nn$, we consider the following computational problem.
	\begin{prob}
		\label{mainpb}
		For an element $u\in\rr^\nn$, compute
		\begin{itemize}
			\item $\displaystyle{\sup_{n\in\nn} u_n}$, \blue{if the supremum is finite};
			\item $\displaystyle{k\in\agmx(u):=\{n\in\nn: u_k=\sup_{m\in\nn} u_m\}}$, \blue{if the set of maximizers is nonempty}.
		\end{itemize}
	\end{prob}
	The objective of this paper is to introduce a theoretical methodology for computing the supremum and an integer maximizer of a real sequence. The main idea developed in this paper is the construction of a formula that provides a finite upper bound for the greatest maximizer of the analyzed sequence. This upper bound allows us to solve Problem~\ref{mainpb} in \blue{a} finite \blue{number of steps}. Indeed, \blue{comparing all terms from index 0 to the integer part of the upper bound suffices} to compute the supremum of the sequence and an integer maximizer. \blue{The smallest possible upper bound must be computed to limit the number of comparisons.}
	
	\blue{The search for an element of an index maximizer of a sequence can be viewed as an abstract integer programming problem. While classical numerical methods (Branch-and-bound, Branch-and-Cut, decompositions, etc.) for integer programming (see e.g.,~\cite{Bestuzheva2025,li2006nonlinear,misener2014antigone,ZHANG2023205}) require finite bounds on the decision variables, real relaxations are useful in particular cases for unconstrained (unbounded) nonlinear integer programs \blue{such} as Problem~\ref{mainpb}. Particular cases \blue{entail imposing} additional conditions on the analyzed sequence. These conditions ensure that the real extension assumes a suitable shape to maximize it. We can impose on the analyzed sequence to be concave-extensible (see, for example, ~\cite{murota2009recent}) (the real extension is concave) or (globally) unimodal (see, for example,~\cite{HASSIN2018795}) (strictly increasing before a real $m$ and strictly decreasing after $m$). Unfortunately, these assumptions are overly restrictive and unsuitable for the problem motivating this article and the examples that we develop later.
	} 
	
	Our motivation for Problem~\ref{mainpb} comes from the study of \emph{generalized peak computation problems} arising in the verification of dynamical systems (see, e.g.,~\cite{adje2015property,miller2021peakdiscrete} in the discrete-time context and~\cite{miller2020peaksafety} in the continuous-time case). Problem~\ref{mainpb} can be viewed as a generalization of \cite{DBLP:journals/jota/Adje21,adje13052025}. A generalized peak computation problem requires two inputs: a dynamical system and an objective function. Generalized peak computation problems encompass several verification problems in dynamical systems. The objective function models the quantity of interest, for which the evolution is analyzed. For example, the objective function can be an  
	energy function~\cite{4303250}, a functional representation of a safe 
	set~\cite{adje2015property,miller2020peaksafety} or a norm~\cite{ahiyevich2018upper,7984147}. Regarding norms, we can find in some references a problem dealing with the computation of an upper bound on the sequences of the norms of the powers of a given matrix (see e.g.,~\cite{dowler2013bounding}). The latter work is not directly motivated by dynamical system verification purposes and improves previous works on the topic based on the Kreiss matrix theorem~\cite{Kreiss1962,Spijker1991,TrefethenEmbree+2020}. References~\cite{ahiyevich2018upper,polyak2018peak,7984147} on norm-based generalized peak computation problems deal with linear systems. Moreover, in~\cite{polyak2018peak,7984147}, the authors only propose lower and upper bounds on the optimal value of the associated optimization problem, whereas in~\cite{ahiyevich2018upper}, an upper bound is proposed. Finally, we find a similar formulation in~\cite{ahmadi2025robust}, where the generalized peak computation problem is called a \emph{robust-to-dynamics optimization} problem. The framework differs from ours in that the authors put an additional constraint. This latter is called an invariance constraint, and it forces the state variables to remain in a set.   
	
	
	In this paper, we address maximization problems for general sequences; thus, the proposed method can be directly applied to generalized peak computation problems in the context of discrete-time systems. The dynamical system input of a generalized peak computation problem is an autonomous time-invariant discrete-time system characterized by a set of initial conditions $\xin$ and a function $T$ that updates the state variables. If we suppose that the size of the vector of state variables is $d$, then we have $\xin\subset \rd$ and $T:\rd \to \rd$. However, the objective function is defined over the state space, that is, a function $\varphi:\rd \to \rr$. A generalized peak computation problem consists in maximizing $\varphi$ over the reachable values set generated from $\xin$ and $T$.
	More formally, a generalized peak computation problem in a discrete-time context is defined from $\xin$, $T$ and $\varphi$ and aims at computing
	\begin{equation}
		\label{infinitepb}
		\sup_{k\in\nn}\sup_{x\in\xin} \varphi(T^k(x))
	\end{equation}
	and, if there exists, an integer $n$ such that 
	\[
	\sup_{x\in\xin} \varphi(T^n(x))=\sup_{k\in\nn}\sup_{x\in\xin} \varphi(T^k(x)).
	\]
	Then writing 
	\begin{equation}
	\label{eq:nudef}
	\forall\,k\in\nn,\ \nu_k:=\sup_{x\in\xin} \varphi(T^k(x)) \text{ and } \nu_{\rm opt}:=\sup_{k\in\nn}\nu_k,
	\end{equation}
	we recover the problem in the form of Problem~\ref{mainpb} where the analyzed sequence is $\nu:=(\nu_k)_{k\in\nn}$.

Despite the personal motivation of the author for generalized peak computation problems, the problem of searching for the maximal term of a sequence also appears as an auxiliary problem for studying well-known sequences. The terminology \emph{maximum excursion value} is used for the famous Syracuse sequences (e.g.,~\cite[Part III]{lagarias2023ultimate} or ~\cite{silva1999maximum}) or in~\cite{prasadestimates} for the Juggler sequences (see~\cite{jugg}). In both cases, this value represents the maximal value of the sequence. Unfortunately, the index that achieves this maximal value (if any) is not considered. In the context of Syracuse sequences and their generalizations, we find the notion of \emph{stopping times}. This notion has similarities with our central notion of truncation index used in this paper to safely truncate the infinite sequence of maximization problems exposed in Equation~\eqref{infinitepb} to a finite sequence.   
	
	Finally, the main contribution of the paper is to propose a general abstract method to solve Problem~\ref{mainpb}. We develop theoretical algorithms to solve Problem~\ref{mainpb} in a finite number of steps. To obtain this number of steps, we construct an upper bound for the greatest integer maximizer. This upper bound is computed from a formula that relies on the upper bounds of the analyzed sequence. These upper bounds are sequences constructed from a sequence $h:=(h_k)_{k\in\nn}$ of strictly increasing and continuous functions on $[0,1]$ and a sequence $\beta:=(\beta_k)_{k\in\nn}$ of elements in $(0,1)$. \blue{If $u$ is the analyzed sequence}, the pair $(h,\beta)$ must satisfy for all $k\in\nn$
	\begin{equation}
	\label{eq:upperbound}
	\blue{u_k\leq h_k(\beta_k^k)}.
	\end{equation}
	For all $k\in\nn$, $h_k$ is strictly increasing and continuous on $[0,1]$ and thus is invertible once it is restricted to $[0,1]$. Then, we can define, for all $k\in\nn$ such that $u_k>h_k(0)$:
	\begin{equation}
	\label{eq:upperboundfct}
	\Fu(k,h,\beta)=\dfrac{\ln(h_{k}^{-1}(u_k))}{\ln(\beta_k)}.
	\end{equation}
	In the paper, we prove that the interesting class of upper bounds consists in eventually decreasing  
	elements $(h,\beta)$ for which the set $\{k\in\nn : u_k>h_k(0)\}\cap [N,+\infty)$ is nonempty, $N$ being the smallest integer from which $(h,\beta)$ is decreasing. We prove that such a pair exists if and only if there exists a term of $u$ greater than the limit superior of $u$. Moreover, the use of eventually decreasing elements $(h,\beta)$ also guarantees that $\Fu(k,h,\beta)$ is an upper bound on the greatest integer of $\agmx(u)$ for all $k\geq N$. We then apply our techniques to well-known sequences: the ratio of two consecutive terms of a Fibonacci sequence, logistic sequences, Syracuse sequences, and the norm-based generalized peak computation problem in the case of discrete-time linear dynamics. We apply the technique developed in~\cite{adje13052025} \blue{based on} quadratic Lyapunov functions and discuss how to \blue{exploit} the sequential approach developed in this paper. Contrary to~\cite{dowler2013bounding,polyak2018peak}, where only lower and upper bounds on the value $\nu_{\rm opt}$ are proposed, \blue{the purpose of this current work is to propose a general abstract method to} compute the exact value of $\nu_{\rm opt}$.

	The paper is organized as follows. Section~\ref{sec:preliminary} is devoted to \blue{the preliminaries including basic results} on the supremum of real sequences\blue{, an overview of the methodology employed, and a discussion of related works}. Section~\ref{mainres} \blue{presents} the main results of the paper. We provide details regarding the useful sets of upper bounds over the analyzed sequence. We also prove that these sets are nonempty in the interesting cases. We \blue{describe} the main properties of the function $\Fu$. Section~\ref{mainres} \blue{ends with} \blue{two} theoretical algorithms \blue{for solving} Problem~\ref{mainpb}. Section~\ref{applis} \blue{is devoted to the} applications of our techniques to concrete sequences. This includes applications to Fibonacci sequences, logistic sequences, Syracuse sequences, and the peak computation problem presented in~\cite{ahiyevich2018upper}.  
	\blue{
	\section{Preliminary discussions}
	\label{sec:preliminary}
	\subsection{Basic results on the supremum of a real sequence}
	We provide useful notations concerning the supremum, the limit superior, and the set of maximizers of a sequence.
		For $u\in\rr^\nn$, we write 
\[		
		\valu{u}:=\sup_{n\in\nn} u_n \text{ and } \ls{u}:=\limsup_{n\to +\infty} u_n.
	\]
	In this paper, we denote by $\Lambda$ the set of real sequences that have a finite supremum as
\[
\Lambda:=\{u\in \rr^\nn : \valu{u}<+\infty\}.
\]
It is easy to see that $u\in \Lambda\iff \ls{u}\in\rr\cup\{-\infty\}$. From this equivalence, we can characterize, for $u\in\Lambda$, $\agmx(u)$ using the tails of $u$ and the elements greater than its limit superior. Hence, we introduce the following notation:
	\[
	 \gsls{u}:=\left\{k\in\nn: u_k>\ls{u}\right\},\ \gs_u:=\left\{k\in\nn: \max_{0\leq j\leq k} u_j>\sup_{j>k} u_j\right\} \text{ and }\ \bigks_u=\inf\gs_u
	\]
	We recall that $\bigks_u<+\infty$ if and only if $\gs_u\neq\emptyset$. 
		\begin{prop}
		\label{argmax}
		If $u\notin \Lambda$ then $\gs_u=\emptyset$ and $\gsls{u}=\emptyset$. If $u\in\Lambda$, then:
		\begin{enumerate}
		\item $\gs_u\neq\emptyset\iff \gsls{u}\neq \emptyset$;
			\item $\gs_u=\emptyset\iff \valu{u}=\ls{u}$.
			\item $\agmx(u)=\emptyset$ or $\agmx(u)$ unbounded $\implies \valu{u}=\ls{u}$;
			\item If $\gs_u\neq\emptyset$ then $\bigks_u=\max\agmx(u)$. 
		\end{enumerate}
	\end{prop}
	}
	\blue{
	\subsection{An overview of the methodology and tools developed in this paper}
	\label{overview}
	\subsubsection{The methodology}
	In this paper, we propose solving Problem~\ref{mainpb} for all sequences $u$ such that the limit superior is not the optimal value of Problem~\ref{mainpb} as required in Proposition~\ref{argmax}. Equivalently, by Proposition~\ref{argmax}, in this case, the analyzed sequence $u$ has a term exceeding its limit superior. The output of our method is the greatest optimal solution of Problem~\ref{mainpb}, $\bigks_u$. Our method can compute the exact optimal value since it computes a maximizer for the sequence $u$.}

\blue{	
The methodology proposed in this paper consists in comparing a finite number of terms of the analyzed sequence to compute its maximum. We must ensure that this finite family of terms contains the maximum of the sequence. 
}

\blue{
For a sequence $u$, when a term $u_k$ is given, we ask the following question: from which index $j$ are the terms of $u$ strictly smaller than $u_k$? First, if the term $u_k$ is strictly greater than the limit superior of $u$, it cannot be the infimum of $u$, and thus this index $j$ exists. Second, $\bigks_u$ is smaller than $j$. Indeed, according to the definition of $j$, $n$ is greater than $j$ implies that $u_n$ is strictly smaller than $u_k$. We can conclude that it is sufficient to compare $u_0,u_1,\ldots,u_j$ to obtain $\valu{u}$.
}

\blue{
The difficulty in the simplicity of the approach lies in computing such an index $j$. This computational problem combines indices and terms of the sequence. Hence, for all sequences $u\in\Lambda$ such that $\gsls{u}$ is nonempty, we need a function capable of converting $u_k$ with $k\in\gsls{u}$ into index $j$ defined above. If $u$ is a positive convergent geometric sequence, the natural logarithm provides a suitable function. Moreover, in this case, $u$ is strictly decreasing and for all $j\geq \ln(u_k)$, we have $u_j<u_k$. However, geometric sequences are extremely restrictive. To be as general as possible, we consider bijections of positive convergent geometric sequences. In addition, these particular sequences serve as upper bounds for the analyzed sequence. This specific form of upper bounds generalizes the approach used in~\cite{adje13052025} for peak computation problems (as recalled in~\eqref{infinitepb}) when $T$ is affine and $\varphi$ is quadratic. In~\cite{adje13052025}, the proposed upper bound has the form $(h(\beta^k))_k$ where $h$ and $\beta$ are constructed from a quadratic Lyapunov function. This construction will be applied to norm-based peak computation problems for linear systems.
}

In this paper, for all $u\in\Lambda$, we prove in Theorem~\ref{mainthmon} that a {\it useful optimal} upper bound $(h(\beta^k))_k$ is available if and only if $\gsls{u}$ is nonempty. Even if this upper bound cannot be constructed in practice, this result responds to the theoretical feasibility of our approach. Because of this impractical side, we accept more general forms as upper bounds, allowing sequences of functions $(h_k)_k$ and of scalars $(\beta_k)_k$ instead of a single function $h$ and a single scalar $\beta$. This leads to the inequalities arising in~\eqref{eq:upperbound}. Finally, we impose, at least, an {\it eventually decreasing} constraint on these sequences to guarantee that they produce an upper bound for the greatest optimal solution of Problem~\ref{mainpb}. 

From the inequalities presented in~\eqref{eq:upperbound}, we can compute a pertinent index $j$ as described above. A pair $(h,\beta)$ allows us to construct the functional $\Fu$ presented in~\eqref{eq:upperboundfct}. This functional provides our desired indices $j$ for a suitable choice of $(h,\beta)$. Then, we can design algorithms to solve Problem~\ref{mainpb} in finite time. The number of comparisons required to solve Problem~\ref{mainpb} depends on the quality of the overapproximation in~\eqref{eq:upperbound}. The inequalities~\eqref{eq:upperbound} must not be too coarse to obtain a small gap between $\Fu(k,h,\beta)$ and the greatest maximizer of Problem~\ref{mainpb}.

\blue{We introduce additional materials to provide further details on the methodology. First, we insist on the fact that to convert the inequalities appearing in~\eqref{eq:upperbound} into indices, the functions $h_k$ must be invertible at least on the open interval $(0,1)$. Moreover, they must preserve some strict inequalities. Hence, we impose that $h_k$ be strictly increasing and surjective on $[0,1]$.  As a strictly increasing function is surjective on an interval if and only if it is continuous on it, we introduce the following set of functions: 
	\[
	\fset:=\{f:\rr\mapsto \rr : f \text{ strictly increasing and continuous on } [0,1]\}.
	\]
Every $f\in\fset$ admits an inverse on $[0,1]$ defined from $[f(0),f(1)]$ to $[0,1]$. For the sake of simplicity, we will denote it by $f^{-1}$. 
}
\blue{
Now, we can formally define the set of eligible sequences as upper bounds for the analyzed sequence. Using the standard notation $\fset^\nn$ for the sequences of elements of $\fset$, we define the following set of upper bounds for $u\in\rr^\nn$:
	\[
	\funset(u):=\left\{(h,\beta)\in \fset^\nn \times (0,1)^\nn: u_k\leq h_k(\beta_k^k),\ \forall\, k\in\nn\right\}. 
	\]
	We sometimes need to separate the functional and scalar parts of an element in $\funset(u)$ and introduce the following two sets:
	\[ 
	\secfunh(u):=\left\{h\in \fset^\nn : \exists\, \beta\in (0,1)^\nn \text{ s.t. } (h,\beta)\in \funset(u)\right\}
	\]
	and for $h\in\secfunh(u)$:
	\[ 
	\secfunb(u,h):=\left\{\beta\in (0,1)^\nn : (h,\beta)\in \funset(u)\right\}.
	\] 
	From the elements of $\funset(u)$, we can derive a functional called {\bf the upper bound functional}.   
	\begin{defi}[Upper bound Functional]
		\label{upperboundfun}
		Let $u\in\rr^\nn$. The upper bound functional is $\Fu:\nn\times \fset^\nn\times (0,1)^\nn\mapsto \rr_+\cup\{+\infty\}$ defined as follows: for all $k\in\nn$, $h\in\fset^\nn$ and $\beta\in (0,1)^\nn$, we have: 
		\begin{equation}
			\label{mainformulaaux}
			\Fu(k,h,\beta):=\left\{
			\begin{array}{lr}
				\dfrac{\ln(h_k^{-1}(u_k))}{\ln(\beta_k)} & \text{ if } (h,\beta)\in \funset(u)\text{ and } u_k>h_k(0)\\
				+\infty & \text{otherwise}
			\end{array}\right.
		\end{equation}
	\end{defi}
	Definition~\ref{upperboundfun} completes the one provided in Equation~\eqref{eq:upperboundfct}. The functional $\Fu$ is set to $+\infty$ for the elements of $\fset^\nn\times (0,1)^\nn$ outside $\funset(u)$ and for the elements of $\funset(u)$ for which $u_k\leq h_k(0)$ \blue{for all $k\in\nn$}. The existence of an integer $k$ such that $u_k>h_k(0)$ fully depends on the choice of the sequence $(h_k)_k$. Consequently, we consider sequences for which there exists an integer $k$ such that $u_k>h_k(0)$. Therefore, we introduce the notion of \textbf{usefulness}.
	\begin{defi}[Useful sequence of strictly increasing continuous functions]
		\label{useful}
		Let $u\in\rr^\nn$. A sequence of functions $h\in \secfunh(u)$ is said to be {\bf useful} for $u$ if the set
		\begin{equation}
			\label{residual}
			\res(u,h):=\{k\in\nn : u_k>h_k(0)\}
		\end{equation}
		is nonempty. 		
		By extension, a pair $(h,\beta)\in\funset(u)$ is {\bf useful} for $u$ if $h$ is useful for $u$.
	\end{defi} 
}
\blue{
From the introduction of our basic tools, we provide a graphical overview of the methodology in Figure~\ref{fig:overview}. This shows how the tools are connected and how they are used in the algorithms. We explain the meaning of each arrow that appears in Figure~\ref{fig:overview} as follows.
}

\begin{figure}[h]
	\fbox{
		\begin{minipage}{0.98\textwidth}
			\begin{center}
				\begin{tikzpicture}
					\node[draw,text width=5.5cm,text centered] at (-3.5,4) {\begin{tabular}{c}$\exists\, (h,\beta)\in \fset\times (0,1)$ s.t.\\ \begin{tabular}{l} $\bullet\, \forall\, k\in\nn,\ u_k\leq h(\beta^k)$ \\ $\bullet\, \exists\, n\in\nn,\ u_n>h(0)$\end{tabular}\end{tabular}};
					\node[draw,text width=2cm,text centered] at (5.5,0.2) {$\Fu(\cdot,h,\beta)$};
					\node[draw,text width=3cm,text centered] at (1.5,-1.5) {Algorithms};
					\node[draw,text width=6cm,text centered] at (1.5,-4) {The greatest optimal solution of Problem~\ref{mainpb}: $\bigks_u$};
					\node[draw,text width=2.5cm,text centered] at (4.7,4) {$\gsls{u}\neq \emptyset$};
					\node[draw,text width=4cm,text centered] at (-3,1.2) {useful $(h,\beta)\in\funset(u)$};
					\draw[ultra thick,line width=1pt,double,<-] (-0.4,4) -- (3,4);
					\draw[ultra thick,line width=1pt,double,->] (-2,1.7) -- (4,3.5);
					\draw[ultra thick,dashed,->] (-3,0.7) -- (1.1,-1);
					\draw[ultra thick,dashed,->] (-0.7,1.1) -- (4.3,0.4);
					\draw[ultra thick,dashed,->] (5.3,-0.3) -- (1.9,-1);
					\draw[ultra thick,->] (1.5,-1.9) -- (1.5,-3.5);
					\draw[ultra thick,dashed,->] (-3,3) -- (-3,1.6);
					\draw (1.3,4) node[above]{\textcircled{\small 1}};
					\draw (1.3,2.6) node[below]{\textcircled{\small 1}};
					\draw (-3,2.3) node[left]{\textcircled{\small 2}};
					\draw (2,0.7) node[above]{\textcircled{\small 3}};
					\draw (-1.4,-0.1) node[below]{\textcircled{\small 4}};
					\draw (3.9,-0.6) node[below]{\textcircled{\small 4}};
					\draw (1.5,-2.6) node[left]{\textcircled{\small 5}};
				\end{tikzpicture}
			\end{center}
		\end{minipage}
	}
	\caption{An overview of the methodology and tools employed to solve Problem~\ref{mainpb}}
	\label{fig:overview}
\end{figure}

\blue{
\noindent \textcircled{\small 1} In Theorem~\ref{mainthmon}, it will be proved that there exists a {\it useful} sequence upper bound if and only if the limit superior of the sequence is not the optimal value of Problem~\ref{mainpb}.
}

\blue{
\noindent \textcircled{\small 2} Theorem~\ref{mainthmon} relies on Theorem~\ref{funatt} to establish the existence of a useful constant upper bound. In Theorem~\ref{funatt}, we prove that the underlying function can be chosen to be affine. However, this affine function cannot be constructed constructed. In practice, as will be shown in the examples, it will be simpler and more natural to construct a non affine function or a non constant sequence as an upper bound of the analyzed sequence. This explains how we allow sequences of functions in $\fset$ rather than a single function in $\fset$ to construct our upper bounds.
}

\blue{
\noindent \textcircled{\small 3} When an element $(h,\beta)$ of $\funset(u)$ is provided, we can construct $\Fu(\cdot,h,\beta)$. However, the functional $\Fu(\cdot,h,\beta)$ must have suitable properties to be integrated in algorithms to solve Problem~\ref{mainpb}. To obtain these desirable properties, the sequences $(h,\beta)$ must be at least {\it eventually decreasing}.   		  
}

\blue{
\noindent \textcircled{\small 4} We propose two algorithms. They differ from the properties of the input $(h,\beta)$; whether they are eventually decreasing (the most general case) or constant. This has an impact on the index $k$ from which we can use $\lfloor\Fu(k,h,\beta)\rfloor$ to compute an upper bound of $\bigks_u$. It also influences the index $k$ for which $\lfloor\Fu(k,h,\beta)\rfloor$ is minimal.
}

\blue{
\noindent \textcircled{\small 5} Our two algorithms are proven to be sound and provide an optimal solution to Problem~\ref{mainpb} in finite time.
}	
\blue{	
	\subsection{Related works}
	\subsubsection{Integer programming approaches}
	Problem~\ref{mainpb}, in its general form, can be viewed as an unconstrained univariate nonlinear integer maximization problem. For such problems, applying existing methods directly from unconstrained integer programming requires either discrete convexity properties (e.g.,  \cite{murota2009recent}) or a unimodal property (e.g., \cite{HASSIN2018795}). In general, nonlinear integer programming problems are computationally difficult to solve without one of these properties (e.g., \cite{Hemmecke2010}). A generic methodology for solving unconstrained nonlinear integer programs is not proposed; instead, specific reformulation approaches tailored to the problem structure are preferred.  
}	

\blue{
	Imposing on our sequence $u$ to be concave-extensible is too restrictive in our case. To be concave-extensible means that we can extend $u$ to a (classical) concave function in $\rr$. This is the same as satisfying a discrete concavity inequality, that is, $u(k+1)+u(k-1)\leq 2u(k)$ for all $k\in\nn$. The examples proposed in Section~\ref{applis} are not concave-extensible. Moreover, with respect to our motivating problem, the sequence $\nu$ defined in~\eqref{eq:nudef} generally fails to have a concave extension. Note that the construction or even the study of the real extension of the integer function $k\mapsto \nu_k$ is a very difficult challenge. Indeed, it boils down to solving functional equations. One possibility to obtain the real extension, for all $x\in\rd$, of $k\mapsto \varphi(T^k(x))$ is to solve a  Schröeder equation (see e.g., \cite{10.11650/twjm/1500407524}) which is, unfortunately, a purely theoretical tool. Precisely, a Schröeder equation (for $T$) is written as $T=H^{-1}\circ M \circ H)$ where the unknowns are $H$ and $M$; $H$ being invertible and $M$ being a full rank matrix. If the Schröeder equation has a solution for all $x\in\rd$, a real extension of $k\mapsto \varphi(T^k(x))$ is $t\mapsto \varphi(H^{-1}(M^t Hx)$; $M^t$ being constructed from a Jordan decomposition. 
}

\blue{
Unimodal (real) functions (see, for example, \cite{HASSIN2018795}) generalize strictly concave functions admitting a maximizer. These are defined as follows. A real function $f$ is said to be unimodal if there exists one real $m$ such that $f$ strictly increases before $m$ and strictly decreases after. This makes $m$ the only local (and thus global) maximizer of $f$. In fact, this approach is used for global optimization problems but can be applied to the real extension of integer programs. For these extensions, the integer maximizer is either $\lfloor m\rfloor$ or $\lfloor m\rfloor+1$. Again, the examples proposed in this paper and the problem motivated therein do not belong to the class of unimodal objective functions.  
}

\blue{	
	We should also mention the class of Mixed Integer Non-Linear Programs (MINLP) (see, for example, \cite{misener2014antigone}). Clearly, the maximization problem associated with~\eqref{infinitepb} falls into this class. The problems in this class are known to be NP-Hard. Hopefully, numerical solvers exist. The current tools require that integer decision variables lie within a feasible bounded set. This is not the case for our maximization problem. However, for a sequence $u$, for well-chosen pairs $(h,\beta)\in\funset(u)$, $\Fu(\cdot,h,\beta)$ computes an upper bound for the integer decision variable of our maximization problem. Therefore, once $\Fu(\cdot,h,\beta)$ is computed, the search for the maximizer can be performed on $[0,\Fu(\cdot,h,\beta)]\cap \nn$. Integer programming solvers can be useful to compute a maximizer and the optimal value of Problem~\ref{mainpb} after obtaining $\Fu(\cdot,h,\beta)$. This complementary approach is left for future works. 
	}
	\blue{
	\subsubsection{On discrete-time peak computation problems}
	Particular discrete-time peak computation problems arise in~\cite{ahiyevich2018upper,dowler2013bounding,polyak2018peak,7984147}. These problems deal with with linear systems. In~\cite{ahiyevich2018upper,dowler2013bounding,7984147}, the authors deal with the maximization of the norms of powers of a given matrix, whereas in~\cite{polyak2018peak}, the maximal absolute value of the terms of an $n$-th order linear system is studied.  In~\cite{ahmadi2025robust}, the authors study discrete-time peak computation problems for switched linear systems and linear objective functions. In~\cite{adje13052025}, affine systems with quadratic objective functions are considered. 
	Discrete-time peak computation problems also appear in a more general form in~\cite{miller2021peakdiscrete}. In these problems, the dynamics $T$ is piecewise polynomial, and the objective function is a polynomial. 
}	
\paragraph{Norm-based peak computation problems on linear systems}
\blue{
In~\cite{ahiyevich2018upper,dowler2013bounding,7984147}, the authors are interested in the computation of the maximum of the sequences of the Euclidean norm of the powers of a given stable matrix. In these papers, the authors propose an upper bound for the optimal peak value. In~\cite{ahiyevich2018upper,7984147}, this upper bound is computed using a semidefinite program solving a discrete Lyapunov equation. In~\cite{dowler2013bounding}, this upper bound is derived a direct computation of the norm of the powers of the matrix using a Schur decomposition of the matrix. Schur decomposition allows us to compute the powers of the matrix by means of the powers of a similar upper triangular matrix. In~\cite{dowler2013bounding,7984147}, a lower bound for the optimal peak value is also proposed. In~\cite{dowler2013bounding}, this lower bound still comes from the Schur decomposition and the powers of the upper triangular matrix mathematically similar to the one analyzed. In~\cite{7984147}, the lower bound is computed for matrices in companion form.
}

\blue{
In~\cite{polyak2018peak}, the authors consider high-order linear real sequences for which the absolute values of the terms is analyzed. In~\cite{polyak2018peak}, the authors propose a closed form for the optimal peak value and the peak instant for the very specific case where the (Schur stable) characteristic polynomial associated with the linear coefficients has all roots that are real and identical. If the roots are real, the authors construct lower and upper bounds on the optimal peak value from lower and upper bounds on the roots of the characteristic polynomial.
}

\blue{
These approaches provide useful information quickly and do not require evaluations of the Euclidean norm of the matrix powers or computing the terms. In the current paper, we apply our technique to the maximization of the Euclidean norm of the matrix powers. The application consists in a slight reformulation of our approach developed in~\cite{adje13052025}. It provides the exact optimal value and the power that reaches the maximum value. Regarding the comparison with~\cite{polyak2018peak}, we can ask whether the upper bound proposed by the authors can be formulated as elements of $\funset(\cdot)$. In future work, we will consider specific derivations of our method for high-order linear systems. 
}

\paragraph{Piecewise linear peak computation problems}
\blue{
In~\cite{ahmadi2025robust}, Ahmadi and Günlük propose to solve peak computation problems when the dynamics $T$ is piecewise linear (i.e., a switched linear system) and has a joint spectral radius strictly less than one. The objective function $\varphi$ is linear. The authors' problem has a supplementary constraint: the states of the system must evolve within an invariant set. The authors compute the optimal value of their peak computation problem in finite time when the invariant is compact. This optimal value is the smallest optimal value of a finite number of semidefinite programs. The authors' approach focuses on numerical methods based on semidefinite programming. The numerical approaches presented in~\cite{DBLP:journals/jota/Adje21,adje13052025} have similarities concerning the use of semidefinite programming and the definition of truncation indices. These approaches differ from those used in our framework. This is mainly due to the presence of the (compact) invariant in~\cite{ahmadi2025robust}. In~\cite{DBLP:journals/jota/Adje21,adje13052025}, compactness is implicit. Indeed, the linear system is stable and thus has a bounded reachable values set. The purpose of the current paper is not to propose a numerical tool. The objective of the paper is to identify sufficient and necessary tools to solve Problem~\eqref{mainpb} in finite time.
}

\paragraph{Polynomial-based peak computation problems}

\blue{	
	Miller et al~\cite{miller2021peakdiscrete} develop an approach based on a discrete Liouville measure equation reformulation of peak computation problems. Specialization to polynomials allows us to reformulate the optimization problem over measures as an optimization problem with linear matrix inequality constraints (duals of semidefinite programs). 
}

\blue{	
The authors' framework integrates invariance and uncertainty into their peak optimization problems. On the one hand, this appears to address a wider class of peak computation problems; on the other hand, the computational aspects require the invariant and uncertainty sets to be compact. Finally, the authors also propose an upper bound for the peak computation problem, which is not necessarily tight. The compactness assumption makes a real difference between our approach and theirs. Their motivation is also different, as they propose a numerical tool to solve peak computation problems.  
Again, the comparison between the current paper and this approach cannot be entirely fair as the motivations differ. Miller et al propose a numerical tool, whereas the purpose of the current paper is a theoretical investigation. Currently, a fully automatic tool is available for solving peak computation problems for stable affine systems with quadratic objective functions~\cite{adje13052025}. In the current paper, the computational aspects are not addressed, and we assume that we know or can compute the exact values of the terms of the analyzed sequence. In practice, in peak computation problems, obtaining the exact values of $\nu_k$ is computationally challenging and relies on closed-form expressions for $\varphi\circ T^n$ or $T^n(\xin)$.
}

	\section{Resolution of Problem~\ref{mainpb} from an upper approximation of the analyzed sequence}
	\label{mainres}
	\blue{
	In this section, we provide technical details concerning the resolution of Problem~\ref{mainpb} from the elements in $\funset(\cdot)$. First, we refine the set of the upper bound candidates by adding an eventually decreasing constraint. Next, we prove that the existence of a useful eventually decreasing upper bound is equivalent to the nonemptyness of $\gsls{\cdot}$. We study the properties of $\mathfrak{F}_{\cdot}$ to guarantee the soundness of the approach. Finally, we present \blue{two} algorithms to solve Problem~\ref{mainpb}. The construction of these algorithms depends on the decreasing property of the chosen pair $(h,\beta)\in\funset(\cdot)$. 
	}
	\subsection{\blue{The set of eventually decreasing upper bound sequences}}	
	\blue{
	We begin our theoretical investigations by proving that $\mathfrak{F}$ is well-defined and satisfies interesting properties in a general setting. First, we prove that the upper bound functional at $k$ is always greater than $k$. This result is important since it implies that $\mathfrak{F}_{\cdot}(\bigks_{\cdot},h,\beta)\geq \bigks_{\cdot}$. 
	}
	
	\blue{
	\begin{prop}
	\label{prop:simpleFuprop}
	Let $u\in\rr^\nn$ and $(h,\beta)\in \funset(u)$ be useful. Let $k\in\res(u,h)$. The following statements hold.
		\begin{enumerate}
			\item $\Fu(k,h,\beta)$ is well-defined and strictly positive if $k\neq 0$ and $\Fu(0,h,\beta)=0$ if and only if $u_0=h_0(1)$;
			\item the equalities hold:
			\[
			\lfloor \Fu(k,h,\beta)\rfloor+1 =\min\{j\in\nn : h_k(\beta_k^j)< u_k\}
			=\min\{j\in\nn : \beta_k^j< h_k^{-1}(u_k)\};
			\]
			\item $\Fu(k,h,\beta)\geq k$.
	\end{enumerate}			
	\end{prop}
}
	\begin{proof}
\blue{
		\emph{1.} Let $(h,\beta)\in \funset(u)$. As $h\in\secfunh(u)$ and we only consider $u_k>h_k(0)$, we have $h_k^{-1}(u_k)>0$. We also have, if $k\in\nn^*$, $u_k<h_k(1)$ meaning that $h_k^{-1}(u_k)\in (0,1)$. If $u_0<h_0(1)$, then the strict inequality $u_k<h_k(1)$ is also valid at $k=0$. As for all $k\in\nn$, $\beta_k\in (0,1)$, $\ln(h_k^{-1}(u_k))$ and $\ln(\beta_k)$ are both strictly negative, and the result holds. It is easy to see that $u_0=h_0(1)$ iff $\Fu(0,h,\beta)=0$.
}

\blue{
\emph{2.} As $(h,\beta)\in \funset(u)$ and $k\in\res(u,h)$, there exists $j\in\nn$ such that $h_k(\beta_k^j)<u_k$. Then the sets
		$\{j\in\nn : h_k(\beta_k^j)< u_k\}$ and $\{j\in\nn : \beta_k^j< h_k^{-1}(u_k)\}$ are equal and nonempty. Let $j\in\nn$ s.t. $\beta_k^j< h_k^{-1}(u_k)$. Applying the natural logarithm and dividing by $\ln(\beta_k)<0$ yields $j> \Fu(k,h,\beta)$. As $j\in\nn$, we have $j\geq \lfloor \Fu(k,h,\beta)\rfloor+1$. 
		Since $\beta_k\in(0,1)$, $\ln(\beta_k)(\lfloor \Fu(k,h,\beta)\rfloor+1)<\Fu(k,h,\beta)\ln(\beta_k)=\ln(h_k^{-1}(u_k))$ and $\beta_k^{\lfloor \Fu(k,h,\beta)\rfloor+1}<h_k^{-1}(u_k)$. Hence $\lfloor \Fu(k,h,\beta)\rfloor+1$ belongs to $\{j\in\nn : \beta_k^j< h_k^{-1}(u_k)\}$ and is its minimum.
}

\blue{
\emph{3.} As $(h,\beta)\in\funset(u)$ and $k\in\res(u,h)$, we have $0<h_k^{-1}(u_k)\leq \beta_k^k<1$. Applying the natural logarithm and dividing by $\ln(\beta_k)<0$ leads to the result.
}
		\end{proof}
	\blue{
	In fact, the set $\funset(\cdot)$ is too large to be used to construct the functional $\Fgen$ even if it is restricted to useful pairs. In Example~\ref{faussexam}, we prove that we always construct an element in $(h,\beta)\in\funset(u)$ for any sequence $u$ with a supremum in $\rr_+^*\cup\{+\infty\}$. \blue{Unfortunately}, we show that the associated functional $\Fu(\cdot,h,\beta)$ fails to systematically produce an upper bound of $\bigks_u$ for all indices $k\in\res(u,h)$.
	}

\blue{
\begin{remark}
		The second statement of Proposition~\ref{prop:simpleFuprop} suggests that $\lfloor\Fu(k,h,\beta)\rfloor+1$ is an extension of the concept of stopping times on positive integers~\cite{Terras1976} or the first passage time~\cite{tao2022almost} that appears in the study of Syracuse sequences. Indeed, $\lfloor\Fu(k,h,\beta)\rfloor+1$ is the smallest integer $j$ such that $h_k(\beta_k^j) < u_k$ whereas $u_k\leq h_k(\beta_k^k)$ as $(h,\beta)\in\funset(u)$.
		\end{remark}
}		
		
	\begin{example}
		\label{faussexam}
		Let $u\in\rr^\nn$ such that $\valu{u}\in\rr_+^\blue{*}\cup \{+\infty\}$. We define for all $k\in\nn$:
		\begin{equation}
			\label{examfond}
			h_k:x\mapsto \left(\frac{|u_k|+1}{|u_k|+0.5}\right)^k|u_k| x \text{ and } \beta_k:=\left(\frac{|u_k|+0.5}{|u_k|+1}\right).
		\end{equation}
		For all $k\in\nn$, the functions $h_k$ are clearly in $\fset$ and $\beta_k\in (0,1)$. Moreover, we have for all $k\in\nn$, $u_k\leq h_k(\beta_k^k)$ and in particular we have $u_k=h_k(\beta_k^k)$ for all $k\in\nn$ such that $u_k\geq 0$. Let us define $h:=(h_k)_k$ and $\beta:=(\beta_k)_k$. We have for all $k\in \nn$, $h_k(0)=0$. As $\valu{u}$ is supposed to be strictly positive, there exists $k\in\nn$ such that $u_k>0$ and so $\res(u,h)\neq \emptyset$. Hence, for all $k\in\nn$ such that $u_k>0$, we have
		\[
		\Fu(k,h,\beta)=\dfrac{\ln h_k^{-1} (u_k)}{\ln \beta_k}=k.  
		\]
		\blue{Note that the third statement of Proposition~\ref{prop:simpleFuprop} holds. However}, we cannot guarantee that $\Fu(k,h,\beta)$ is an upper bound for $\bigks_u$ for all $k\blue{\in\res(u,h)}$.
		\qed
	\end{example}
	In Example~\ref{faussexam}, we see that we must restrict the choice of upper bound candidates. Next, we will see that a good choice is to pick {\bf eventually decreasing} elements of $\funset(u)$.
	
	We use the standard functional order on $\fset$, that is, for $f,g\in\fset$, $f\leq g \iff f(x)\leq g(x)$ for all $x\in [0,1]$. We also use the standard weak order on $\fset\times (0,1)$, that is,  for all $(f,\beta), (g,\gamma)\in \fset\times (0,1)$, $(f,\beta)\leq (g,\gamma) \iff f\leq g $ and $\beta\leq \gamma$. 
	
Then, \blue{for $u\in\rr^\nn$, we introduce} the set $\fmonfung(u)$ of eventually decreasing elements of $\secfunh(u)$, that is,
	\[
	\fmonfung(u):=\{h\in\secfunh(u): \exists\, N\in \nn\ \forall\, n\geq N,\ h_{n+1}\leq h_n\}.
	\]
	For $h\in\fmonfung(u)$, we denote the smallest index from which $h$ is decreasing by $\mon(h)$. \blue{Clearly, the set $\fmonfung(u)$ also contains the decreasing (from index 0) elements of $\secfunh(u)$. From our notation, the decreasing elements of $\secfunh(u)$ are the elements of $\fmonfung(u)$ that satisfy $\mon(h)=0$.}
	
	\blue{Now, considering both functional and scalar parts, we denote by $\fmong(u)$ the set of} all eventually decreasing elements of $\funset(u)$, that is,
	\[
	\fmong(u):=\{(h,\beta)\in\secfunh(u): \exists\, N\in \nn\ \forall\, n\geq N,\ (h_{n+1},\beta_{n+1})\leq (h_n,\beta_n)\}.
	\]
	For $(h,\beta)\in \fmong(u)$, we denote the smallest index from which $h$ and $\beta$ are decreasing by $\mon(h,\beta)$. \blue{Again, $\fmong(u)$ contains the decreasing elements of $\funset(u)$ and which are the elements of $\fmong(u)$ such that $\mon(h,\beta)=0$.}
	
\blue{As mentioned in Subsection~\ref{overview}, we will prove that if $\gsls{u}$ is nonempty, we can construct a useful constant sequence that provides an upper bound of $u$. In the sequel, we allow eventually constant sequences instead of constant sequences. Thus, we denote by $\fcste(u)$ the set of {\bf eventually constant pairs} in $\funset(u)$}:
\[
\fcste(u):=\{(h,\beta)\in \funset(u) : \exists\, (g,\gamma,N)\in\fset\times (0,1) \times\nn\ \forall\, n\geq N,\  (h_n,\beta_n)=(g,\gamma)\}.
\]
For $(h,\beta)\in\fcste(u)$, we denote by $\cst(h,\beta)$ the smallest index from which the sequence $(h,\beta)$ is constant. \blue{Similarly to $\fmonfung(u)$ and $\fmong(u)$, the constant (from index 0) sequences of $\funset(u)$ are the elements of $\fcste(u)$ such that $\cst(h,\beta)=0$.}

	\begin{remark}
		We insist on the fact that in the definition of being eventually decreasing, both $h$ and $\beta$ are decreasing after $\mon(h,\beta)$. Let $n$ be greater than $\mon(h,\beta)$. As the functional order is used for $h$, we have $h_{n+1}(\beta_{n+1}^{n+1})\leq h_n(\beta_{n+1}^{n+1})$. Since $h_n$ is strictly increasing and $\beta_{n+1}\in (0,1)$, we have $h_n(\beta_{n+1}^{n+1})< h_n(\beta_{n+1}^n)$. Finally, as $\beta_{n+1}\leq \beta_n$ and $h_n$ is increasing, we get $h_n(\beta_{n+1}^n)\leq h_n(\beta_n^n)$. We conclude that $h_{n+1}(\beta_{n+1}^{n+1})<h_n(\beta_n)$. 
		
		We could ask if it suffices only to require the weaker condition: $h_{\blue{n+1}}(\beta_{n+1}^{\blue{n+1}})\leq h_n(\beta_n^n)$ for all $n\geq \mon(h,\beta)$. However, in the definition of the upper bound functional (Definition~\ref{upperboundfun}), $(h_n)_{n\in\nn}$ is decoupled from $(\beta_n)_{n\in\nn}$. The weaker condition would not be sufficient to compare $\Fu(k,h,\beta)$, $\Fu(\bigks_u,h,\beta)$ and $\bigks_u$ for $k\leq \bigks_u$. 
	\end{remark}
	
	We extend the usefulness to the eventually decreasing elements. Now, we require that $\res(u,h)$ contains an integer greater than $\mon(h)$.   
	\begin{defi}[Usefully decreasing]
		Let $u\in\rr^\nn$. An element $h\in\fmonfung(u)$ (resp. $(h,\beta)\in \fmong(u)$) is {\bf usefully decreasing} if and only if $\res(u,h)\cap [\mon(h),+\infty)\neq \emptyset$ (resp. $\res(u,h)\cap [\mon(h,\beta),+\infty)\neq \emptyset$).
	\end{defi}
	An (\blue{decreasing}) element \blue{$h\in\fmonfung(u)$ such that $\mon(h)=0$ (resp. $(h,\beta)\in \fmong(u)$ such that $\mon(h,\beta)=0$)} is usefully decreasing if and only if it is useful \blue{according to} Definition~\ref{useful}. 
	
	We will see in the algorithms developed in Subsection~\ref{subsec:algo} that decreasing pairs will be easier to manage practically than eventually decreasing  pairs. It is obvious that decreasing pairs provide eventually decreasing pairs. Now, we prove that we can always construct a decreasing pair from an eventually decreasing pair. Later, we will prove that this procedure can increase the value of $\Fu$. 
	\begin{prop}
		\label{mongtomon0}
		Let $u\in\rr^\nn$. The following equivalence holds: 
		\[\fmong(u) \neq \emptyset\iff \blue{\exists\, (g,\gamma)\in\fmong(u) \text{ such that } \mon(g,\gamma)=0}.\] Moreover, there exists $(h,\beta)\in\fmong(u)$ usefully decreasing if and only if there exists a useful $(g,\gamma)\in\fmong(u)$ such that $\mon(g,\gamma)=0$. 
	\end{prop}
	\begin{proof}
		Suppose there exists $(h,\beta)\in\fmong(u)$. Let us write $N=\blue{\mon(h,\beta)}$. Then, we can define for all integers $n\blue{<} N$, $\blue{g_n=\max_{i=0,\ldots,N} h_i}$ and $\blue{g_n=h_n}$ for all integers $n\blue{\geq}N$. Similarly, we define, for all $n\blue{<} N$, $\blue{\gamma_n=\max_{i=0,\ldots,N} \beta_i}$ and for all $n\blue{\geq}N$, $\blue{\gamma_n=\beta_n}$. Then \blue{$(g,\gamma)\in\fmong(u)$ and $\mon(g,\gamma)=0$}. \blue{The reverse implication is clear.}
		
		Clearly, if there exists a \blue{useful $(g,\gamma)\in\fmong(u)$ such that $\mon(g,\gamma)=0$}, then it is also eventually decreasing and usefully decreasing. Now suppose that there exists $(h,\beta)\in\fmong(u)$ usefully decreasing, the construction of $(g,\gamma)$ from above is useful as for some $n\geq \mon(h,\beta)$, $u_n>g_n(0)=h_n(0)$.
	\end{proof}
	
	\subsection{Existence of useful upper bound sequences}
	We have presented the important sets of upper bound sequences and their basic properties. Before developing the main properties of the upper bound functional $\Fu$ defined in Definition~\ref{upperboundfun}, we show how the existence of useful constant elements and usefully decreasing elements are related to the nonemptiness of $\gsls{u}$.
	
	First, as mentioned in Example~\ref{faussexam}, the set $\funset(u)$ is always nonempty for all $u\in\rr^\nn$. Second, we can easily characterize $\Lambda$ from $\secfunh(u)$ and describe a simple situation in which $u_0$ is the maximum of $u$.
	\begin{prop}
		\label{nonemptyfun}
		Let $u\in\rr^\nn$. The following statements hold:
		\begin{enumerate}
			\item $\funset(u)\neq \emptyset$.
			\item $u\in\Lambda$ if and only there exists $h\in\secfunh(u)$ such that the set $\{h_k(1) : k\in\nn\}$ is bounded from above.
			\item if there exists \blue{$h\in\fmonfung(u)$ such that $\mon(h)=0$} and $u_0=h_0(1)$ then $u\in\Lambda$ and $u_k<u_0$ for all $k\in\nn^\blue{*}$.
		\end{enumerate}
	\end{prop}
	\begin{proof}
		\itshape{1.} 
		The pair $(h,\beta)$ proposed in~\eqref{examfond} belongs to $\funset(u)$ even if $u\notin\Lambda$ \blue{or} $\valu{u}$ is not strictly positive.
		
		\itshape{2.} Let us take $h\in\secfunh(u)$ such that $\{h_k(1): k\in\nn\}$ is bounded from above.\blue{Then, we get } $u_k\leq h_k(\beta_k^k) < h_k(1)$. \blue{This} implies that $\valu{u}\leq \sup \{h_k(1): k\in\nn\}<+\infty$. On the contrary, if $u\in\Lambda$, then the constant sequence $h$ defined from $f:x\mapsto x+\valu{u}$ belongs to $\secfunh(u)$. \blue{Moreover, we have} $\{h_k(1): k\in\nn\}=\{\valu{u}+1\}$.
		
		\itshape{3.} Assume that there exists $h\in\secfunh(u)$ decreasing such that $u_0=h_0(1)$. Then for some $\beta\in (0,1)^\nn$, for all $k\in\nn^*$, $u_k\leq h_k(\beta_k^k) < h_k(1)\leq h_0(1)=u_0$.
	\end{proof}
	\begin{remark}
		\label{zeroisopt}
		The third statement of Proposition~\ref{nonemptyfun} implies, if it holds, that the term $u_0$ is the supremum of the analyzed sequence $u$. Moreover, index 0 is the only maximizer of the sequence.
	\end{remark}
	
	We can go further than the results of Proposition~\ref{nonemptyfun} to completely characterize $\Lambda$ by using the eventually decreasing upper bounds.
	\begin{prop}
		\label{nonuppbound}
		Let $u\in\rr^\nn$. Then $u\in\Lambda\iff \fmonfung(u)\neq \emptyset$.
	\end{prop}
	
	\begin{proof}
		Assume that $u\in\Lambda$. The constant sequence function in the proof of the second statement of Proposition~\ref{nonemptyfun} belongs to $\fmonfung(u)$. Now suppose that $u\notin\Lambda$ and $\fmonfung(u)\neq \emptyset$. Using the second statement of Proposition~\ref{nonemptyfun}, for all $h\in\secfunh(u)$, the set $\{h_k(1):k\in\nn\}$ is unbounded from above. Taking $g\in\fmonfung(u)$, we have for all $k\in\nn$, $g_{k}(1)\leq \max \{g_i(1): i=0,\ldots,\mon(g)\}\in\rr$. This contradicts the second statement of Proposition~\ref{nonemptyfun} and thus, $\fmonfung(u)$ is empty.
	\end{proof}
	
	Theorem~\ref{funatt} is the main result of this section. First, it justifies the finiteness of the upper bound functional $\Fu$ for $u\in\Lambda$ with $\gsls{u}\neq \emptyset$. Second, it serves to construct constant and non constant elements in $\funset(u)$. Finally, the existence of a useful and optimal function is important for the optimality of $\Fu$.
	
	\begin{theorem}[Existence of a useful optimal affine function] 
		\label{funatt}
		Let $u\in\Lambda$ such that $\gsls{u}\neq \emptyset$. Then, there exist $a>0$, $b\in (0,1)$ and $c\in\rr$ such that:
		\begin{enumerate}
			\item $\{k\in\nn : u_k>c\}\neq \emptyset$;
			\item $u_k\leq a b^k+c$ for all $k\in\nn$;
			\item $u_{\bigks_u}=ab^{\bigks_u}+c$.
		\end{enumerate}
	\end{theorem}
	
	\begin{proof}
		If $\gsls{u}$ is nonempty then $u_{\bigks_u}=\valu{u}>\ls{u}\in\rr\cup\{-\infty\}$ and for all $t\in (\ls{u},\valu{u})$ there exists 
		$k\in\nn$, such that $u_k>t$. So let $c$ be in $(\ls{u},\valu{u})$. Now as $c>\ls{u}$, there exists $N_c\in\nn$ s.t. $\sup_{k\geq N_c} u_k\leq c$. This is equivalent to $u_k\leq c$ for all $k\geq N_c$. We conclude that for all $k\geq N_c$, $u_k\leq c+ab^k$ for all $a>0$ and $b\in (0,1)$. Note that as $c<\valu{u}$ and $\bigks_u$ is the greatest element of the argmax of $u$, then $N_c$ is strictly greater than $\bigks_u$. Now, let us introduce the following set: 
		\[
		{AB}_c:= \left\{(a,b)\in\rr_+^* \times (0,1): u_{\bigks_u}=\valu{u}=a b^{\bigks_u}+c\right\}.
		\] 
		It is easy to see that ${AB}_c$ is nonempty, indeed, for all $b$ in $(0,1)$, $(\frac{u_{\bigks_u}-c}{b^{\bigks_u}},b)\in {AB}_c$. Moreover, if $\bigks_u>0$, for all $(a,b)\in {AB}_c$,
		for all $0<k< \bigks_u$, $u_k\leq u_{\bigks_u}=a b^{\bigks_u}+c<a b^k+c$ as $b\in (0,1)$. We also have $u_0\leq u_{\bigks_u}=ab^{\bigks_u}+c\leq a+c$ as $b\in (0,1)$. Hence, the result $u_k\leq ab^k+c$ holds on the sets $[0,\bigks_u]\cap \nn$ and $[N_c,+\infty)\cap\nn$ for every $(a,b)\in {AB}_c$. It suffices to find $(a,b)\in {AB}_c$, such that $u_k\leq ab^k+c$ for all $k\in (\bigks_u,N_c]\cap \nn$. Recall that we have $\bigks_u+1\leq N_c$.
		
		Let us write $(-\infty,0)\ni\gamma:=\max_{k\in (\bigks_u,N_c]\cap\nn} (u_{\bigks_u}-u_k)/(\bigks_u-k)$. This cannot be zero since $\bigks_u$ is the maximal integer $k$ such that $u_k=\max_{n\in\nn} u_n$. Now, let $(a,b)\in \rr_+\times (0,1)$ and define $f_{a,b}:\rr\ni y\mapsto a b^y+c$. The function $f_{a,b}$ is derivable, strictly convex and $f'_{a,b}(y)=ab^y\ln(b)$ for all $y\in\rr$. If $f'_{a,b}(\bigks_u)=\gamma$ and $(a,b)\in {AB}_c$ then, by the convexity of $f_{a,b}$, the tangent of $f$ at $y=\bigks_u$ satisfies for all $k\in\nn$, $f(k)\geq f'(\bigks_u)(k-\bigks_u)+f(\bigks_u)$ which is the same as $ab^k+c\geq \gamma (k-\bigks_u)+u_{\bigks_u}$. \red{From the} definition of $\gamma$, we have for all $k\in (\bigks_u,N_c\red{]\cap \nn}$, $(\bigks_u-k)\gamma\leq u_{\bigks_u}-u_k$ and so $u_k\leq u_{\bigks_u}+(k-\bigks_u)\gamma$. We conclude that if $f'_{a,b}(\bigks_u)=\gamma$ and $(a,b)\in {AB}_c$, $u_k\leq ab^k+c$ for all $k\in (\bigks_u,N_c]\red{\cap\nn}$. Finally, we look for $(a,b)\in \rr_+\times (0,1)$ such that:
		\[
		f_{a,b}(\bigks_u)=u_{\bigks_u} \text{ i.e. } (a,b)\in {AB}_c\text{ and } f'_{a,b}(\bigks_u)=\gamma.
		\]
		From $(a,b)\in {AB}_c$, we get $ab^{\bigks_u}=u_{\bigks_u}-c$. Hence, $f'_{a,b}(\bigks_u)=\gamma$ if and only if $\ln(b)=\frac{\gamma}{u_{\bigks_u}-c}<0$. So $b=\exp(\frac{\gamma}{u_{\bigks_u}-c})\in (0,1)$ and $a=(u_{\bigks_u}-c)\exp(-\frac{\bigks_u\gamma}{u_{\bigks_u}-c})>0$.
	\end{proof}
	
	In the proof of Theorem~\ref{funatt}, we exhibit the parameters of the affine function. Unfortunately (but naturally), all the parameters depend on $\bigks_u$ and $u_{\bigks_u}$ \red{which} is exactly what we want to compute. Theorem~\ref{funatt} allows \red{us} to construct an affine function in $\secfunh(u)$ \red{that} is optimal. Optimal means that at the index $\bigks_u$, the upper bound $ab^{\bigks_u}+c$ is \blue{tight i.e.,} equal to $u_{\bigks_u}$. Hence, we obtain $\Fu(\bigks_u,x\mapsto ax+c,b)=\bigks_u$.
	
	\begin{coro}
		\label{maincoro}
		Let $u\in\Lambda$ such that $\gsls{u}\neq \emptyset$. Then, there exists a pair $(h,\beta)\in\fmong(u)$  such that $u_{\bigks_u}=h_{\bigks_u}(\beta^{\bigks_u})$.
	\end{coro}
	
	Theorem~\ref{funatt} proves that the nonemptiness of $\gsls{u}$ is a sufficient condition for the existence of a usefully decreasing element in $\fmong(u)$. Now, we also prove the necessity. The necessity requires an auxiliary result, \red{which is} formulated in Lemma~\ref{simpleproperty}.
	\begin{lemma}
		\label{simpleproperty}
		Let $u\in\Lambda$. Let $(h,\beta)\in\fmong(u)$ be usefully decreasing. Then:
		\begin{itemize}
			\item the sequence $(\beta_k^k)_{k\in\nn}$ converges to 0 strictly decreasingly from $\mon(h,\beta)$;
			\item for all $n\geq \mon(h,\beta)$, $\ls{u}\leq h_n(0)$.
		\end{itemize}
	\end{lemma}
	
	\begin{proof}
		Let us write $N=\mon(h,\beta)$. For all $n\geq N$, $0<\beta_{n+1}\leq \beta_n\leq \beta_N<1$ and thus $0<\beta_{n+1}^{n+1}<\beta_{n+1}^n\leq \beta_n^n\leq \beta_N^n$. Taking the limit over $n$ leads to $\ls{\beta}=\lim_{n\to +\infty} \beta_n^n=0$. 
		
		Now let us write, $\omega:=(h_k(\beta_k^k))_{k\in\nn}$. Let $n\geq N$ and define $\zeta:=(h_n(\beta_k^k))_k$. For all $k\geq n$, as $(h,\beta)\in\funset(u)$ and $(h_m)_{m\geq N}$ decreases, $u_k\leq \omega_k\leq \zeta_k$. Taking the $\limsup$ over $k$ and as $h_n$ is (upper semi-)continuous, we get $\ls{u}\leq \ls{\omega}\leq \ls{\zeta}\leq h_n(\ls{\beta})=h_n(0)$. 
	\end{proof}
	\blue{Now, we can completely characterize the nonemptiness of $\gsls{u}$ from the existence of a usefully decreasing elements in $\fmong(u)$. As shown in Figure~\ref{fig:overview}, the methodology that we develop can always be based on inequalities of the form~\eqref{eq:upperbound} and on the upper bound functional $\Fu$.}
	\begin{theorem}
		\label{mainthmon}
		Let $u\in\Lambda$. Then $\gsls{u}\neq \emptyset$ if and only if there exists $(h,\beta)\in\fmong(u)$ usefully decreasing. 
	\end{theorem}
	\begin{proof}
		Theorem~\ref{funatt} asserts that $\fcste(u)$ is nonempty when $\gsls{u}$ is nonempty, \red{which implies sufficiency}. The necessity is a consequence of Lemma~\ref{simpleproperty}. Indeed, if there exists $(h,\beta)\in\fmong(u)$ usefully decreasing, by definition, we can pick $n\in\res(u,h)\cap [\mon(h,\beta),+\infty)$. Therefore, $u_n>h_n(0)\geq \ls{u}$ meaning that $n\in\gsls{u}$.
	\end{proof}
	Finally, we can ask \red{whether} we \red{can choose} non constant elements in $\funset(u)$ rather than the affine function furnished by Theorem~\ref{funatt}. The answer is positive, and the construction of the sequence is based on Lemma~\ref{existence}.  
	\begin{lemma}
		\label{existence}
		Let $u\in\Lambda$ such that $\gsls{u}\neq \emptyset$. Let us consider the triplet $(a,b,c)\in \rr_+^*\times (0,1)\times \rr$ constructed in Theorem~\ref{funatt}.
		Let us consider:
		\begin{enumerate}
			\item a decreasing sequence $\gamma:=(\gamma_k)_{k\in\nn}$ in $(c,\valu{u})$;
			\item a decreasing sequence $\beta:=(\beta_k)_{k\in\nn}$ in $[b,1)$;
			\item a decreasing sequence $\alpha:=(\alpha_k)_{k\in\nn}\in \fset^\nn$ such that :
			\begin{equation}
				\label{eqalpha}
				\inf_{k\in\nn}\inf_{x\in (0,1]} \dfrac{\alpha_k(x)}{x}\geq a .
			\end{equation}
		\end{enumerate}
		Let us define for all $k\in\nn$, $h_k:[0,1]\ni x\mapsto a_k(x)+\gamma_k$ and $h=(h_k)_{k\in\nn}$. Then, $(h,\beta)\in\fmong(u)$ \blue{with $\mon(h,\beta)=0$} and $(h,\beta)$ is useful.
		
	\end{lemma}
	
	\begin{proof}
		We have from for all $k\in\nn$, $\beta_k\in [b,1)$, $\alpha_k(\beta_k^k)\geq a\beta_k^k\geq a b^k$. Thus, from Lemma~\ref{funatt}, 
		$h_k(\beta_k^k)\geq ab^k+c\geq u_k$ and $(h,\beta)\in\funset(u)$. By assumptions on $\alpha$ and $\gamma$, $h$ is clearly decreasing, and we conclude that $h\in\fmonfung(u)$. Finally, $(h,\beta)\in\fmong(u)$. Moreover, since for all $k\in\nn$ $h_k(0)=\gamma_k<\valu{u}=u_{\bigks_u}$, $\bigks_u\in\res(u,h)$ and $h$ is useful.
	\end{proof} 
	
	\begin{prop}
		\label{monfun}
		Let $u\in\Lambda$ such that $\gsls{u}\neq \emptyset$. There exists \blue{a useful $(h,\beta)\in\fmong(u)$ satisfying $\mon(h,\beta)=0$ such that} $h$ and $\beta$ are not constant sequences.
	\end{prop}
	\begin{proof}
		It suffices to apply Lemma~\ref{existence} with for all $k\in\nn$, $\alpha_k:x\mapsto x(k+1)^{-1}+a$, $\beta_k=\min\{b+(k+1)^{-1},(b+1)/2\}$ and  $\gamma_k=\min\{c+(k+1)^{-1},(\valu{u}+c)/2\}$.
	\end{proof}
	\blue{Another} possible sequences $\alpha_k$ in Lemma~\ref{existence} can be chosen. \red{For example}, any sequence of functions of the form $\alpha_k:x\mapsto f(k)g(x)+a$ where $f$ is a decreasing strictly positive function and $g$ is a strictly increasing continuous positive function such that $g(x)\geq x^t$ with $t>1$.
	
	\subsection{Main properties of the upper \red{bound functional} $\Fu$}
	\label{mainpropsec}
	Theorem~\ref{mainthmon} and Proposition~\ref{monfun} prove that we can find $(h,\beta)\in\fmong(u)$ when $\gsls{u}$ is nonempty. Hence, $\Fu$ is finite somewhere when $\gsls{u}$ is nonempty. Now, we present the main properties of the upper bound functional $(k,h,\beta)\to\Fu(k,h,\beta)$: the increasing properties of $\Fu$ with respect to $h$ and $\beta$ and the fact that $\Fu(k,h,\beta)$ provides upper bounds for $\bigks_u$.  
	
	\subsubsection{\blue{Increasing} properties of the upper bound functional $\Fu$}
	The \red{functional} $(h,\beta)\to \Fu(\cdot,h,\beta)$ relies on the \blue{sequence of inverses of functions $h_k\in\fset$}. Hence, to study \red{the} increasing properties of $\Fu(\cdot,h,\beta)$ with respect to $h$ we \red{provide} some details on the decreasing property of the application $\fset\ni f\to f^{-1}$. First, we pay attention to the fact that we can only compare the inverses \blue{of elements in $\fset$} on the values at the intersection of their range. \blue{Recall that} for all sets $J$ and all \blue{$\{f_j,j\in J\}\subset \fset$}, we have
\[		
		\blue{\bigcap_{j\in J}[f_j(0),f_j(1)]=\left[\sup_{j\in J} f_j(0),\inf_{j\in J}f_j(1)\right]}.
	\]
	\blue{We will also study the behavior of $\Fu$ with respect to the minima and maxima of finite families of elements of $\fset$. Thus, we develop formulae to compute} the inverses of the minima and maxima of a finite family of elements of $\fset$. For a finite family \blue{$\{f_i: i\in I\}\subset \fset$}, we write
	\begin{equation}
		\label{minmaxfct}
		\blue{\underline{f}_I:[0,1]\ni x\mapsto \min_{i\in I} f_i(x)}\quad \text{ and } \quad 
		\blue{\overline{f}_I:[0,1]\ni x\mapsto \max_{i\in I} f_i(x)}.
	\end{equation}

\blue{
When computing the inverse of $\underline{f}_I$ (resp. $\overline{f}_I$) at some real $y\in \underline{f}_I([0,1])$ (resp. $y\in \underline{f}_I([0,1])$), we must pay attention to the functions $f_i$ for which $y<f_i(0)$ (resp. $f_i(1)<y$). More precisely, these functions must be removed from the computations as the inverse of $f_i$ is not defined at this $y$. Moreover, since the functions $f_i$ and their inverse are strictly increasing, the inverse of "max" becomes "min" and the inverse of "min" is "max".  	
}
	\begin{lemma}
		\label{invcompare}
		Let \blue{$f_1,f_2\in \fset$} and \blue{$\{g_i:i\in I\}$} be a finite family in $\fset$. The following statements hold.
		\begin{enumerate}
			\item If $\blue{\displaystyle{\mathop{\cap}_{i=1,2}}[f_i(0),f_i(1)]}=\emptyset$, \blue{$f_1< f_2\iff f_1(0)<f_2(0)$}.
			\item If $\blue{\displaystyle{\mathop{\cap}_{i=1,2}}[f_i(0),f_i(1)]}\neq\emptyset$, $f_1\leq f_2\iff \forall\, y\in\blue{\displaystyle{\mathop{\cap}_{i=1,2}}[f_i(0),f_i(1)]}$,  $f_2^{-1}(y)\leq f_1^{-1}(y)$.
			\item  The functions \blue{$\underline{g}_I,\overline{g}_I$} belong to $\fset$ and for all $y\in \blue{[\underline{g}_I(0),\underline{g}_{I}(1)]}$ and all $z\in \blue{[\overline{g}_I(0),\overline{g}_{I}(1)]}$:
			\[\blue{\underline{g}_I^{-1}(y)=\max_{\{i\in I:\, g_i(0)\leq y\}} g_i^{-1}(y)}\quad\text{ and
			}\quad \blue{\overline{g}_I^{-1}(z)=\min_{\{i\in I:\, z\leq g_i(1)\}} g_i^{-1}(z)}.\]
		\end{enumerate}
	\end{lemma}
	
	\begin{proof}
	\blue{For $f_1,f_2\in\fset$, we simply write $F:=[f_1(0),f_1(1)]\cap [f_2(0),f_2(1)]$.}
	
		{\itshape 1}. Suppose that \blue{$F=\emptyset$}. If \blue{$f_1(x) < f_2(x)$} for all $x\in [0,1]$, it obviously implies that \blue{$f_1(0)<f_2(0)$}. Now, suppose that \blue{$f_1(0)<f_2(0)$}. As \blue{$F=\emptyset$}, we have for all $x\in [0,1]$ \blue{$f_1(x)\notin [f_2(0),f_2(1)]$}. Thus either, \blue{$f_1(0)>f_2(1)$} or  \blue{$f_1(1)<f_2(0)$}. Then, we exclude \blue{$f_1(0)>f_2(1)$} as \blue{$f_1(0)<f_2(0)$} and \blue{$f_2$} is strictly increasing on $[0,1]$. We conclude that \blue{$f_1(x) \leq f_1(1)< f_2(0)\leq f_2(x)$} for all $x\in [0,1]$ as \blue{$f_1$ and $f_2$ }are strictly increasing on $[0,1]$.
		
		{\itshape 2}. Suppose that \blue{$F\neq \emptyset$}. Assume that \blue{$f_1\leq f_2$} on $[0,1]$. Let \blue{$y\in F$}. As \blue{$y\in [f_2(0),f_2(1)]$}, \blue{$y=f_2(x)$} for some $x\in [0,1]$. Moreover, by assumption, we have \blue{$f_1(x)\leq f_2(x)$}. As \blue{$y\in [f_1(0),f_1(1)]$}, we can apply \blue{$f_1^{-1}$} which is strictly increasing to $f_2(x)=y$. We get \blue{$x\leq f_1^{-1}(f_2(x))$}. As \blue{$x=f_2^{-1}(y)$}, we conclude that \blue{$f_2^{-1}(y)\leq f_1^{-1}(y)$}. 
		
		Now assume that for all \blue{$y\in F, f_2^{-1}(y)\leq f_1^{-1}(y)$}. Let $x\in [0,1]$. We write \blue{$z=f_1(x)$} and thus \blue{$z\in [f_1(0),f_1(1)]$}. Suppose that \blue{$z\in F$}, then \blue{$f_2^{-1}(z)\leq f_1^{-1}(z)$}. Applying \blue{$f_2$}, we \red{obtain} \blue{$z\leq f_2(f_1^{-1}(z))$} which is the same as \blue{$f_1(x)\leq f_2(x)$}. Now suppose that \blue{$z\notin F$} meaning that \blue{$z\notin [f_2(0),f_2(1)]$}. If \blue{$z<f_2(0)$}, then we have \blue{$f_1(x)<f_2(0)\leq f_2(x)$} as \blue{$f_2$} increases on $[0,1]$. Assume now that \blue{$z=f_1(x)>f_2(1)$}. As \blue{$F$} is nonempty, there exists $u\in [0,x)$ such that \blue{$f_1(u)\leq f_2(1)$}. By \red{the} continuity of \blue{$f_1$}, there exists $v\in [u,x)$ such that \blue{$f_1(v)=f_2(1)$}. As \blue{$f_1(v)\in F$}, we have \blue{$f_2^{-1}(f_1(v))\leq f_1^{-1}(f_1(v))=v$} and thus \blue{$f_2^{-1}(f_2(1))=1\leq v<x\leq 1$}. This implies that $z$ cannot \red{be} greater than \blue{$f_2(1)$}. This ends the proof. 
		
		{\itshape 3}. The minimum and maximum of a finite family of strictly increasing continuous functions \red{are} still strictly increasing and continuous; hence \blue{$\underline{g}_I,\overline{g}_I\in\fset$}. Let \blue{$y\in [\underline{g}_I(0),\underline{g}_I(1)]$}. Then there exists $x\in [0,1]$ such that \blue{$y=\underline{g}_I(x)$}. As $I$ is finite, there exists $\ell\in I$ such that \blue{$\underline{g}_I(x)=g_\ell (x)$. It is clear that $y=g_\ell(x)\geq g_\ell(0)$}. Now let \blue{$i\in I$ be such that $g_i(0)\leq y$}. Then \blue{$y\in [g_i(0),\underline{g}_I(1)]\blue{=[g_i(0),g_i(1)]\cap [\underline{g}_I(0),\underline{g}_I(1)]}$}. From the second statement, \blue{$g_i^{-1}(y)\leq \underline{g}_I^{-1}(y)$}. Finally, \blue{$\max_{\{i\in I:\, g_i(0)\leq y\}} g_i^{-1}(y)\leq  \underline{g}_I^{-1}(y)=x=g_\ell^{-1}(y)\leq \max_{\{i\in I:\, g_i(0)\leq y\}} g_i^{-1}(y)$} and we conclude that \blue{$\underline{g}_I^{-1}(y)=\max_{\{i\in I:\, g_i(0)\leq y\}} g_i^{-1}(y)$}. The proof for \blue{$\overline{g}_I^{-1}$} is similar.
	\end{proof}
	
	The set $\fset^\nn$ is naturally endowed with \red{the} standard lattice operations of \red{the} sequence spaces. We recall that for $g,h\in\fset^\nn$, $g\leq h$ iff $g_n(x)\leq h_n(x)$ for all $n\in\nn$ for all $x\in [0,1]$. Hence, \blue{we extend} the notation \red{in}~\eqref{minmaxfct} \blue{to} the minimum and maximum of \blue{finite families in $\fset^\nn$. To avoid confusion, for a finite family in $\fset^\nn$, we surround the indices of the family in parentheses, whereas the indices of the terms of the sequences are written as usual. Let $\{h_{(i)},i\in I\}$ be a finite family in $\fset^\nn$. We extend the notation used in~\eqref{minmaxfct} for the minimum and maximum of $\{h_{(i)},i\in I\}$ as follows}:
	{ \everymath={\displaystyle}
		\[
		\begin{array}{lc}
			&\underline{h}_{(I)}:=\min_{i\in I} h_{(i)}=\left([0,1]\ni x \to \min_{i\in I} h_{(i),k}(x)\right)_{k\in\nn}=\left(\underline{h}_{(I),k}\right)_{k\in\nn}\\
			\text{ and } & \overline{h}_{(I)}:=\max_{i\in I} h_{(i)}=\left([0,1]\ni x\to \max_{i\in I} h_{(i),k}(x)\right)_{k\in\nn}=\left(\overline{h}_{(I),k}\right)_{k\in\nn}.
		\end{array}
		\]
		Moreover, for a finite family $\{\beta_{(i)} : i\in I\}\subset (0,1)^\nn$, we denote by:
		\[
		\underline{\beta}_{(I)}:=\left(\min_{i\in I} \beta_{(i),k}\right)_{k\in\nn} \text{ and } \overline{\beta}_{(I)}:=\left(\max_{i\in I} \beta_{(i),k}\right)_{k\in\nn}
		\] 
		respectively, the minimum and maximum of $\{\beta_{(i)} : i\in I\}$.
	}
	
	It can be observed that any element of $\fset^\nn$ greater than an element of $\secfunh(u)$ belongs to $\secfunh(u)$. \red{In addition}, the set function $\secfunh(u)\ni h\to \secfunb(u,h)$ is order-preserving with respect to the order defined on $\fset^\nn$. Moreover, as the \red{elements} of $\fset$ are increasing functions, any element of $(0,1)^\nn$ greater than an element of $\secfunb(u,h)$ belongs to $\secfunb(u,h)$. \red{These} simple facts are presented in Lemma~\ref{factlattice}.
	\begin{lemma}
		\label{factlattice}
		Let $u\in\rr^\nn$ and $h\in\secfunh(u)$. 
		\begin{enumerate}
			\item Let $\red{g}\in\fset^\nn$ such that $\red{h\leq g}$, then $\red{g}\in\secfunh(u)$ and $\red{\secfunb(u,h)\subseteq\secfunb(u,g)}$.
			\item Let $\beta\in\secfunb(u,h)$. Then all $\beta'\in (0,1)^\nn$ such that $\beta\leq\beta'$ belong to $\secfunb(u,h)$.
		\end{enumerate}
	\end{lemma}
	\begin{remark}
		We should mention that the two statements of Lemma~\ref{factlattice} refer to a well-known notion from ordered \red{set} theory. More precisely, $\secfunh(u)$ and $\secfunb(u,h)$ are said to be upward closed (or upset or order filter). The interested reader can \blue{refer to~\cite{MR1058437,MR2446182}}. 
	\end{remark}
	\begin{prop}
		\label{secfunprop}
		Let $u\in\rr^\nn$. The following assertions hold:
		\begin{itemize}
			\item For all finite sets $I$, $\left\{(h_{(i)},\beta_{(i)}): i\in I\right\}\subset \funset(u)$:
			\[
			\overline{\beta}_{(I)}\in \bigcap_{i\in I} \secfunb(u,h_{(i)}).
			\]
			
			\item Let $I$ be finite and $\{h_{(i)}: i\in I\}\subset \secfunh(u)$. Then $\underline{h}_{(I)}$ and $\overline{h}_{(I)}$ belong to $\secfunh(u)$ and
			\[
			\secfunb\left(u,\underline{h}_{(I)}\right)=\bigcap_{i\in I} \secfunb\left(u,h_{(i)}\right)\subseteq \secfunb\left(u,\overline{h}_{(I)}\right).
			\]
			If, moreover, for all $i\in I$, $h_{(i)}\in\fmonfung(u)$, then $\underline{h}_{(I)}$ and $\overline{h}_{(I)}$ belong to $\fmonfung(u)$ and $\mon(\underline{h}_{(I)})$ and $\mon(\overline{h}_{(I)})$ are smaller than $\displaystyle{\max_{i\in I} \mon(h_{(i)})}$.
		\end{itemize}
	\end{prop}
	
	\begin{proof}
		The first statement is a direct consequence of the first assertion in Lemma~\ref{factlattice}.
		
		Let us prove the second statement. First, $\underline{h}_{(I)}$ and $\overline{h}_{(I)}$ belong $\fset^\nn$ from the third statement of Lemma~\ref{invcompare}. The fact that $\overline{h}_{(I)}$ belongs to $\secfunh(u)$ is a consequence of the first assertion in Lemma~\ref{factlattice}. Now as for all $i\in I$, $h_{(i)}\in\secfunh(u)$, we have for all $i\in I$, for all $k\in\nn$, $u_k\leq h_{(i),k}(\beta_{k}^k)$ for all $\beta\in \cap_{i\in I} \secfunb\left(u,h_{(i)}\right)$. Hence, for all $k\in\nn$, $u_k\leq \underline{h}_{(I),k}(\beta_{k}^k)$ for all $\beta\in \cap_{i\in I} \secfunb(u,h_{(i)})$. This proves that $\underline{h}_{(I)}\in \secfunh(u)$ and $\cap_{i\in I} \secfunb(u,h_{(i)})  \subseteq \secfunb(u,\underline{h}_{(I)})$. The converse inclusion is a consequence of the first statement of Lemma~\ref{factlattice}. Now suppose that $h_{(i)}\in\fmonfung(u)$ for all $i\in I$. Then, we have, for all $i\in I$, for all $x\in [0,1]$, $h_{(i),k}(x)\leq h_{(i),k+1}(x)$ for all $k\geq \mon(h_{(i)})$. This implies that for all $x\in [0,1]$, $\underline{h}_{(I),k}(x)\leq \underline{h}_{(I),k+1}(x)$ for all $k\geq \max_{i\in I} \mon(h_{(i)})$ and $\underline{h}_{(I)}\in\fmonfung(u)$ with $\max_{i\in I} \mon(h_{(i)})\geq \mon(\underline{h}_{(I)})$. \blue{Similar arguments allow us to show the results for $\overline{h}_{(I)}$.}
	\end{proof}
	Lemma~\ref{factlattice} and Proposition~\ref{secfunprop} show the \red{increasing} properties of $\secfunb(u,h)$ with respect to $h$. From \red{these} results, we can derive \red{the} increase properties of $\Fu$ with respect to $h$ and $\beta$.  
	\begin{prop}
		\label{simplefactFu}
		Let $u\in\rr^\nn$. The following properties hold:
		\begin{enumerate}
			\item Let $h\in\secfunh(u)$ be useful and $k\in\res(u,h)$. Let $\left\{\beta_{(i)},i\in I\right\}$ be a finite family in $\secfunb(u,h)$. Then \[\Fu\left(k,h,\underline{\beta}_{(I)}\right)\leq \min_{i\in I}\Fu\left(k,h,\beta_{(i)}\right).\]
			\item Let $\left\{h_{(i)}: i\in I\right\}$ be a finite family in $\secfunh(u)$ such that at least one element is useful. Let $k\in \cup_{i\in I}\res\left(u,h_{(i)}\right)$. Let $\beta\in\secfunb\left(u,\underline{h}_{(I)}\right)$. Then 
			\[
			\Fu\left(k,\underline{h}_{(I)},\beta\right)\leq \min_{i\in I} \Fu\left(k,h_{(i)},\beta\right) .
			\]
		\end{enumerate}
	\end{prop}
	
	\begin{proof}
		({\itshape 1}) The result follows from the fact that the function $(0,1)\ni s\mapsto \ln(h_k^{-1}(u_k))/\ln(s)$ is increasing as $\ln(h_k^{-1}(u_k))$ is negative.
		
		({\itshape 2}) First, from Proposition~\ref{secfunprop}, $\underline{h}_I\in \secfunh(u)$ and $\beta\in\cap_{i\in I} \secfunb(u,h_{(i)})$. Second, from Lemma~\ref{invcompare}, for all $k\in\nn$, \blue{${\underline{h}}_{(I),k}^{-1}(u_k)=\max_{\{i\in I: h_{(i),k}(0)\leq u_k\}} h_{(i),k}^{-1}(u_k)$}. Note that \blue{the set $\{i\in I: h_{(i),k}(0)\leq u_k\}$} is nonempty as it contains all $i\in I$ such that $k\in\res\left(u,h_{(i)}\right)$. Moreover, ${\underline{h}}_{(I),k}^{-1}(u_k)\geq h_{(j),k}^{-1}(u_k)$ for all $j\in I$ such that $h_{(j),k}(0)<u_k$ and so ${\underline{h}}_{(I),k}^{-1}(u_k)>0$. Let $i\in I$. If $k\notin \res(u,h_{(i)})$, $\Fu(k,h_{(i)},\beta)=+\infty$. Now assume that $k\in \res(u,h_{(i)})$, $h_{(i),k}^{-1}(u_k)\leq {\underline{h}}_{(I),k}^{-1}(u_k)$. By applying $\ln$ and dividing by $\ln(\beta_k)<0$, we obtain the desired result.
	\end{proof}
	
	\begin{coro}
		\label{coroinc}
		Let $u\in\rr^\nn$. Let $h\in\funset(u)$. Then, for all useful $g\in\funset(u)$ such that $g \leq h$, we have for all $k\in\res(u,g)$ and all $\beta\in\secfunb(u,g)$:
		\[
		\Fu(k,g,\beta)\leq \Fu(k,h,\beta).
		\]
	\end{coro}
	
	The first statement of Proposition~\ref{simplefactFu} \red{should} be read as follows. \red{When} a sequence $h\in\secfunh(u)$ has been found, we \red{should} choose the smallest possible $\beta\in\secfunb(u,h)$ to obtain the smallest possible value of $\Fu(k,h,\beta)$. The second statement has a similar meaning in a functional sense. If we have at hand several functions in $\secfunh(u)$ sharing \red{the} same sequence $\beta$, we \red{must} consider the minimum of these functions to obtain the smallest possible value of $\Fu(k,h,\beta)$.
	\begin{remark}
		In the proof of Proposition~\ref{mongtomon0}, the pair \blue{$(g,\gamma)\in\fmong(u)$ such that $\mon(g,\gamma)=0$ constructed from $(h,\beta)\in\fmong(u)$ is greater than $(h,\beta)$}. In this case, as proved in Proposition~\ref{simplefactFu}, \red{using} \blue{$(g,\gamma)$ instead of $(h,\beta)$} \red{increases} the value of $\Fu(k,\cdot,\cdot)$ i.e. \blue{$\Fu(k,h,\beta)\leq \Fu(k,g,\gamma)$} for all $k\in\res(u,g)\cap [0,\blue{\mon(h,\beta)}]$.
	\end{remark}
	\subsubsection{The upper bound functional $\Fu$ as an upper bound on $\bigks_u$}
	Here, we prove that using eventually and usefully decreasing pairs $(h,\beta)$ \red{allows us} to prove that $\Fu(k,h,\beta)$ is an upper bound \red{on} $\bigks_u$ for all $k\geq \mon(h,\beta)$. First, we prove some useful properties satisfied by $\Fu$. This particularizes the results of Proposition~\ref{prop:simpleFuprop} to eventually decreasing pairs.
	\begin{theorem}
		\label{mainprop}
		Let $u\in\rr^\nn$ and \blue{$(h,\beta)\in \fmong(u)$ be usefully decreasing}. Let \blue{$k\geq\mon(h,\beta)$}. The following statements hold.
		\begin{enumerate}
			\item For all $j>\Fu(k,h,\beta)$, $u_j<u_k$.
			\item Let $j\in [\mon(h,\beta),k]\cap \nn$. If $u_j\leq u_k$ then $\Fu(k,h,\beta)\leq \Fu(j,h,\beta)$. 
\item			
			\blue{If moreover, $(h,\beta)\in\fcste(u)$, the condition $j\leq k$ can be dropped, that is $j,k\geq\cst(h,\beta)$ and $u_j\leq u_k$ implies that $\Fu(k,h,\beta)\leq \Fu(j,h,\beta)$.
			}
		\end{enumerate}
	\end{theorem}
	\begin{proof}	
		\emph{1.} Let $j>\Fu(k,h,\beta)$ then \blue{from the first statement of Proposition~\ref{prop:simpleFuprop}} $j\geq \lfloor \Fu(k,h,\beta)\rfloor+1> k$ and, as $h_k$ is increasing, $h_k(\beta_k^j)\leq h_k(\beta_k^{\lfloor \Fu(k,h,\beta)\rfloor+1})<u_k$ from the second statement \blue{of Proposition~\ref{prop:simpleFuprop}}. Moreover, by definition, as $(h,\beta)\in\fmong(u)$ and $j>k\geq \mon(h,\beta)$, we have $h_j(x)\leq h_k(x)$ for all $x\in [0,1]$ and $0<\beta_j\leq \beta_k<1$. Then, $h_j(\beta_j^j)\leq h_k(\beta_j^j)$ and as $h_k$ is increasing $h_k(\beta_j^j)\leq h_k(\beta_k^j)$. It follows that $u_j\leq h_j(\beta_j^j)<u_k$.
		
		\emph{2.} Let $j \red{\in} [\mon(h,\beta),k]\cap \nn$ and assume that $u_j\leq u_k$. If $j\notin \res(u,h)$, the result holds as $\Fu(j,h,\beta)=+\infty$. Now suppose that $j\in\res(u,h)$. As $k  \red{\geq} j\geq \mon(h,\beta)$, we have $h_k(x)\leq h_j(x)$ for all $x\in [0,1]$. This implies that $h_k(0)\leq h_j(0)$ and $h_k(1)\leq h_j(1)$. As $j\in\res(u,h)$, we have $h_j(0)<u_j$ and as $h\in\secfunh(u)$, we have $u_k< h_k(1)$. By assumption, $u_j\leq u_k$ then $u_j,u_k\in \blue{[h_k(0),h_k(1)]\cap [h_j(0),h_j(1)]}=[h_j(0),h_k(1)]$. From Lemma~\ref{invcompare}, we have $h_j^{-1}(u_j)\leq h_k^{-1}(u_j)$.  Hence, as $\ln$ and $h_k^{-1}$ are increasing, $\ln (h_j^{-1}(u_j))\leq \ln (h_k^{-1}(u_j))\leq \ln(h_k^{-1}(u_k))$. Moreover, as $\beta$ is decreasing from $\mon(h,\beta)$, we get $\ln(\beta_j)/\ln(\beta_k)\in (0,1)$. Dividing by $\ln(\beta_j)<0$ we have $\Fu(j,h,\beta)\geq \ln(h_k^{-1}(u_k)/\ln(\beta_j)\geq  (\ln(h_k^{-1}(u_k))/\ln(\beta_j))\times (\ln(\beta_j)/\ln(\beta_k))=\Fu(k,h,\beta)$. 
		
		\emph{3.} Assume that there exists $(g,\gamma)\in\fset\times (0,1)$ such that $(h_n,\beta_n)=(g,\gamma)$ for all $n\geq \cst(h,\beta)$. The conclusion follows from the fact that $[g(0),g(1)]\ni x\to (\ln (g^{-1}(x)))/\ln(\gamma)$ is decreasing  as $\gamma\in (0,1)$.
	\end{proof}
	
	\begin{remark}
 When $(h,\beta)\in\fmong(u)$, \blue{the second statement of Theorem~\ref{mainprop} confirms that $\lfloor \Fu(k,h,\beta)\rfloor+1$ is a stopping integer. Indeed, after} $\lfloor\Fu(k,h,\beta)\rfloor+1$, the terms of $u$ are strictly smaller than $u_k$. Without the auxiliary sequence characterized by $(h,\beta)$, it is not possible to establish such results.
	\end{remark}
	
	\begin{coro}
		\label{mainpropmon}
		Let $u\in\Lambda$ such that $\gsls{u}\neq \emptyset$. Let \blue{$(h,\beta)\in \fmong(u)$ be useful and such that $\mon(h,\beta)=0$}. Then,
		\begin{enumerate}
			\item for all $j\in\nn$ such that $j>\Fu(k,h,\beta)$, $u_j<u_k$;
			\item let $j\in\res(u,h)$. If $j\leq k$ and $u_j\leq u_k$ then $\Fu(k,h,\beta)\leq \Fu(j,h,\beta)$. If moreover, \blue{$(h,\beta)\in \fcste(u)$ and such that $\cst(h,\beta)=0$}, $u_j\leq u_k\implies \Fu(k,h,\beta)\leq \Fu(j,h,\beta)$.
		\end{enumerate}
	\end{coro}
	\begin{remark}
		We insist on the fact that if $(h,\beta)\in \fmong(u)$ and not in $\fcste(u)$, we cannot guarantee that $\Fu(k,h,\beta)\leq \Fu(j,h,\beta)$ when $u_j\leq u_k$ and $j>k$. This means \red{that} we may find some $j>\bigks_u$ such that $\Fu(j,h,\beta)$ is strictly smaller than $\Fu(\bigks_u,h,\beta)$.  
	\end{remark}
	
	\begin{prop}
		\label{inter}
		Let $u\in\rr^\nn$ and $(h,\beta)\in\funset(u)$. The following statements hold:
		\begin{enumerate}
			\item For all $j\geq \bigks_u$, $\bigks_u\leq \Fu(j,h,\beta)$;
			\item Assume that $(h,\beta)\in \fmong(u)$ and is usefully decreasing, then:
			\begin{enumerate}
				\item for all $j\geq \mon(h,\beta)$, we have $\bigks_u\leq \Fu(j,h,\beta)$;
				\item if $\bigks_u\geq \mon(h,\beta)$, we have for all $j\in [\mon(h,\beta),\bigks_u]$:
				\[
				\Fu(\bigks_u,h,\beta)\leq \Fu(j,h,\beta) .
				\] 
			\end{enumerate}
			\item If, moreover $(h,\beta)\in \fcste(u)$ (and is still usefully decreasing) and $\bigks_u\geq \cst(h,\beta)$, then we have: 
			\[
			\min_{k\in\res(u,h)\cap [\mon(h,\beta),+\infty)} \Fu(k,h,\beta)=\Fu(\bigks_\nu,h,\beta)\enspace .
			\]
		\end{enumerate}
	\end{prop}
	
	\begin{proof}
		({\itshape 1}) Let $j\geq \bigks_u$. If $j\notin \res(u,h)$, then $\Fu(j,h,\beta)=+\infty$ and the inequality holds. Now, assume that $(h,\beta)$ is useful and $j\in\res(u,h)$. From the third statement of \blue{Proposition~\ref{prop:simpleFuprop}}, $\Fu(j,h,\beta)\geq j\geq \bigks_u$. 
		
		({\itshape 2}) (\itshape{a}) Let assume that $(h,\beta)\in\fmong(u)$. Let $j\geq \mon(h,\beta)$. If $j\notin \res(u,h)$, \red{then} the inequality \red{holds trivially}. Now, assume that $j \in\res(u,h)$. If $\bigks_u>\Fu(j,h,\beta)$, from the \blue{first} statement of Theorem~\ref{mainprop}, $u_{\bigks_u}<u_j$. This contradicts that $u_{\bigks_u}=\valu{u}$.
		
		(\itshape{b}) Let us suppose that $\bigks_u\geq \mon(h,\beta)$. Let $j\in [\mon(h,\beta),\bigks_u]$. If $j\notin\res(u,h)$, \red{then} the inequality holds. Suppose that $j\in\res(u,h)$. Then, as $u_j\leq u_{\bigks_u}=\valu{u}$ and $j\leq  \bigks_u$, from the last statement of Theorem~\ref{mainprop}, we have $\Fu(\bigks_u,h,\beta)\leq \Fu(j,h,\beta)$. 
		
		({\itshape 3}) First, as $\fcste(u)\subseteq \fmong(u)$, $\mon(h,\beta)$ is defined for all $(h,\beta)\in \fcste(u)$. As $(h,\beta)$ can be decreasing before being constant, we have $\mon(h,\beta)\leq \cst(h,\beta)$. Let $(g,\gamma)\in \fset\times (0,1)$ such that $(h_n,\beta_n)=(g,\gamma)$ for all $n\geq \cst(h,\beta)$. We prove that $\bigks_u\in\res(u,h)$. As $(h,\beta)$ is usefully decreasing, 
		$\res(u,h)\cap [\mon(h,\beta),+\infty)$ is nonempty. If there exists $j\in \res(u,h)\cap [\mon(h,\beta),+\infty)$ such that $j<\bigks_u$, then, as $j\geq \mon(h,\beta)$ and $\bigks_u\in\agmx(u)$,
		we have $u_{\bigks_u}\geq u_j>h_j(0)\geq h_{\bigks_u}(0)$. If there exists $j\in \res(u,h)\cap [\mon(h,\beta),+\infty)$ such that $j>\bigks_u$, then, as $j>\bigks_u\geq \cst(h,\beta)$ and $\bigks_u\in\agmx(u)$, we have $u_{\bigks_u} \geq u_j>h_j(0)=g(0)=h_{\bigks_u}(0)$. 
		
		Now let $k\in\res(u,h)\cap [\mon(h,\beta),+\infty)$. If $k\leq \bigks_u$, then as $\bigks_u\geq \mon(h,\beta)$, the assertion (b) of the second statement can be used, and we have  $\Fu(\bigks_u,h,\beta)\leq \Fu(k,h,\beta)$. Now suppose that $k>\bigks_u$. We can use the \blue{third} statement of Theorem~\ref{mainprop} as $k>\bigks_u\geq \cst(h,\beta)$, $u_k< u_{\bigks_u}$, $(h,\beta)\in\fcste(u)$ and $\bigks_u\in \res(u,h)$.   
	\end{proof}
	
	\begin{coro}
		\label{intercst}
		Let $u\in\Lambda$ \blue{be such that} $\gsls{u}\neq \emptyset$. \blue{Let $(h,\beta)\in\fmong(u)$ be useful and such that $\mon(h,\beta)=0$}. The following statements hold:
		\begin{enumerate}
			\item For all $j\in\nn$, $\bigks_u\leq \Fu(j,h,\beta)$.
			\item For all $j\leq \bigks_u$, $\Fu(\bigks_u,h,\beta)\leq \Fu(j,h,\beta)$.
			\item If, moreover \blue{$(h,\beta)$ belongs $\fcste(u)$ and such that $\cst(h,\beta)=0$}, then we have: 
			\[
			\min_{k\in\res(u,h)} \Fu(k,h,\beta)=\Fu(\bigks_\nu,h,\beta)\enspace .
			\]
		\end{enumerate}
	\end{coro}
	The third statement of Corollary~\ref{intercst} can be directly \red{applied} in a particular context. Suppose that we have an explicit closed formula for $u_k$, that is, there exists a function $f:\rr_+\mapsto \rr$ for all $k\in\nn$ such that $u_k=f(k)$. Suppose that we have \blue{$(h,\beta)\in\fcste(u)$ with $\cst(h,\beta)=0$. It} is identified as an element of $\fset\times (0,1)$. The function $h^{-1}$ \red{can} be used to regularize and strictly convexify, meaning that $(\ln(\beta))^{-1}\times \ln\circ h^{-1}\circ f$ is \red{a} smooth and strictly convex \red{function}. \red{Thus, we} obtain with some calculus the exact value of $\bigks_\nu$ as the integer part (or integer part +1) of the zero of the derivative of $\rr_+\ni s\to \ln(h^{-1}(f(s)))/(\ln(\beta))$.    
	
	We \red{end} the presentation of the main properties of $\Fu$ with a complete characterization of $\bigks_u$ from the upper bound \red{functional} $\Fu$.
	\begin{theorem}[Formula Optimality]
		\label{formulaopt}
		For all $u\in\rr^\nn$, we have the following formula
		\[
		\bigks_u = \inf_{(h,\beta)\in\fmong(u)}\ \inf_{\substack{k\in\res(u,h)\\ k\geq \mon(h,\beta)}} \Fu(k,h,\beta)=  \inf_{(h,\beta)\in\fmong(u)}\ \inf_{\substack{k\in\res(u,h)\\ k\geq \mon(h,\beta)}}\dfrac{\ln(h_{\red{k}}^{-1}(u_k))}{\ln(\beta)}
		\]
		and the $\inf$ can be replaced by $\min$ when $\gsls{u}\neq \emptyset$.
	\end{theorem}
		\begin{proof}
			Let us define 
			\[
			\ff:=\inf_{(h,\beta)\in\fmong(u)}\ \inf_{\substack{k\in\res(u,h)\\ k\geq \mon(h,\beta)}} \Fu(k,h,\beta).
			\]
			
			If $u\notin \Lambda$, then from Proposition~\ref{nonuppbound}, $\fmong(u)=\emptyset$ and thus $\ff=+\infty$. Moreover, from Proposition~\ref{argmax}, $\bigks_u=+\infty$.
			
			Suppose that $u\in \Lambda$. From Proposition~\ref{nonuppbound}, $\fmong(u)\neq\emptyset$. 
			Assume that $\gs_u=\emptyset$. This is the same as $\bigks_u=+\infty$. From Proposition~\ref{argmax}, $\gsls{u}=\emptyset$ and from Theorem~\ref{mainthmon}, there does not exist a usefully decreasing element i.e. for all $(h,\beta)\in\fmong(u)$, $\res(u,h)\cap [\mon(h,\beta),+\infty)=\emptyset$ which leads to $\ff=+\infty$. Assume that $\gs_u\neq\emptyset$ and thus $\gsls{u}\neq \emptyset$. From the second statement of Proposition~\ref{inter}, $\bigks_u\leq \ff$. Moreover, from Corollary~\ref{maincoro}, there exists \blue{$(h,\beta)\in\fmong(u)$ such that $\mon(h,\beta)=0$} and $u_{\bigks_u}=h_{\bigks_u}(\beta_{\bigks_u}^{\bigks_u})>h_{\bigks_u}(0)$ and $\ff<+\infty$. Taking the inverse of $h_{\bigks_u}$ on $[0,1]$ and the natural logarithm, we get $\bigks_u=\Fu(\bigks_u,h,\beta)$. By definition of $\ff$, $\ff\leq\Fu(\bigks_u,h,\beta) =\bigks_u$. We conclude that $\bigks_u=\ff$.
		\end{proof}
		\subsection{Resolution of Problem~\ref{mainpb} using the upper bound functional $\Fu$}
		\label{subsec:algo}
		As the upper bound \red{functional} $\Fu$ provides an upper bound of $\bigks_u$ for a given $u\in\Lambda$ with $\gsls{u}\neq \emptyset$, we can solve Problem~\ref{mainpb} in finite time \red{by} truncating $u$ \blue{using} $\lfloor\cdot\rfloor\circ \Fu$. This is formally described in Theorem~\ref{summary3}. The definition of the truncation index depends on the \blue{monotonicity properties} of the upper bound over $u$. 
		
		\begin{theorem}[Obtaining a solution to Problem~\ref{mainpb}]
			\label{summary3}
			Let $u\in\Lambda$ with $\gsls{u}\neq \emptyset$. Let $(h,\beta)\in\fmong(u)$ be \red{usefully} decreasing for $u$. Then :
			\[
			\max_{n\in\nn} u_n=\max\{u_n : n= 0,\ldots,\mon(h,\beta),\ldots,\min_{k\in\res(u,h)\cap [\mon(h,\beta),+\infty)}\lfloor\Fu(k,h,\beta)\rfloor\}.
			\]
			\comment{
			If moreover, \blue{$\mon(h,\beta)=0$}, then:
			\[
			\max_{n\in\nn} u_n=\max\{u_n : n= 0,\ldots,\min_{k\in\res(u,h)}\lfloor\Fu(k,h,\beta)\rfloor\}.
			\]  
			}
			If, finally, \blue{$(h,\beta)\in\fcste(u)$ with $\cst(h,\beta)=0$}, then:
			\[
			\max_{n\in\nn} u_n=\max\{u_n : n= 0,\ldots,\lfloor\Fu(\bigks_u,h,\beta)\rfloor\}.
			\]  
		\end{theorem}
		\begin{proof}
			\blue{A direct application of Proposition~\ref{inter} and Corollary~\ref{intercst}.}
		\end{proof} 
		
		We present \blue{two} theoretical algorithms based on Theorem~\ref{summary3}. These algorithms compute $\valu{u}$ and the smallest element of $\agmx(u)$. Algorithm~\ref{algomaingen} describes the most general situation, \blue{in which} the computations are \red{performed} from an upper bound $(h,\beta)\in\fmong(u)$. \blue{This includes the case \blue{in which} the upper bound $(h,\beta)$ satisfies $\mon(h,\beta)=0$. In this case, Algorithm~\ref{algomaingen} can be simplified}.  Algorithm~\ref{algomaincst} deals with an upper bound $(h,\beta)\in\fcste(u)$ with $\cst(h,\beta)=0$. In the presented algorithms, the upper bounds $(h,\beta)$ are supposed to be usefully decreasing \blue{(equivalent to being useful when $\mon(h,\beta)=0$ or $\cst(h,\beta)=0$)}.
		
		\begin{algorithm}[h!]
			\DontPrintSemicolon
			
			\Input{$u\in\Lambda$ with $\gsls{u}\neq \emptyset$ and $(h,\beta)\in\fmong(u)$ usefully decreasing} 
			\Output{$\valu{u}$ and $\bigkse_u=\min\agmx(u)$}
			\Begin{
				$k=0$; $K=+\infty$, $\vmax=-\infty$, $\kmax=0$\;
				\While{$k<\mon(h,\beta)$}
				{
					\If{$\vmax< u_k$}{
						$\vmax=u_k$\;
						$\kmax=k$\;
					}
					$k=k+1$\;
				}
				\While{$k\leq K$}
				{
					\If{$k\in\res(u,h)$}{
						$K=\min\{K,\lfloor\Fu(k,h,\beta)\rfloor\}$\;
						\If{$\vmax< u_k$}{
							$\vmax=u_k$\;
							$\kmax=k$\;
						}
					}
					$k=k+1$\;
				}
				$\valu{u}=\vmax$\;
				$\bigkse_u=\kmax$\;
			}
			\caption{Resolution of Problem~\ref{mainpb} with an eventually and usefully decreasing upper bound}
			\label{algomaingen}
		\end{algorithm}
		
		\blue{Following the second statement of Proposition~\ref{inter}, in} Algorithm~\ref{algomaingen}, to compute a truncation index using $\Fu$, we \red{must} go over $\mon(h,\beta)$ to ensure that we compute an upper bound of $\bigks_u$. Once $\mon(h,\beta)$ \red{is} achieved, $\Fu(k,h,\beta)$ is guaranteed to be greater than $\bigks_u$. However, we do not know for which $k$ greater than $\mon(h,\beta)$, $\Fu(k,h,\beta)$ is minimal. \red{Updating} the truncation index $K$ \red{reduces} the number of iterations. In this context, $u_k$ is \red{assumed to be easily computable}. In \blue{certain} situations, the computation of $u_k$ can be extremely expensive (e.g., the optimal value of an optimization problem~\cite{adje13052025}). Finally, the integer maximizer found is the minimum of $\agmx(u)$, as the update condition at Line 11 in Algorithm~\ref{algomaingen} is $\vmax<u_k$. To get $\bigks_u$, it suffices to replace $\vmax<u_k$ by $\vmax\leq u_k$. \blue{Note that the while loop from Line 3 to Line 7 in Algorithm~\ref{algomaingen} is not executed when $\mon(h,\beta)=0$. Thus, after initialization, Algorithm~\ref{algomaingen} starts directly at Line 8 when $\mon(h,\beta)=0$.}
	\comment{	
		\begin{algorithm}[h!]
			\DontPrintSemicolon
			
			\Input{$u\in\Lambda$ with $\gsls{u}\neq \emptyset$ and $(h,\beta)\in\fmon(u)$ useful} 
			\Output{$\valu{u}$ and $\bigkse_u=\min\agmx(u)$}
			\Begin{
				$k=0$; $K=+\infty$, $\vmax=-\infty$, $\kmax=0$\;
				\While{$k\leq K$}
				{
					\If{$k\in\res(u,h)$}{
						$K=\min\{K,\lfloor\Fu(k,h,\beta)\rfloor\}$\;
						\If{$\vmax< u_k$}{
							$\vmax=u_k$\;
							$\kmax=k$\;
						}
					}
					$k=k+1$\;
				}
				$\valu{u}=\vmax$\;
				$\bigkse_u=\kmax$\;
			}
			\caption{Resolution of Problem~\ref{mainpb} with a useful decreasing upper bound}
			\label{algomaindec}
		\end{algorithm}
		\blue{Following the second statement of Proposition~\ref{inter},} in Algorithm~\ref{algomaindec}, we can start computing a truncation index from $\Fu$ from $k=0$ as $\Fu(k,h,\beta)$ is guaranteed to be greater than $\bigks_u$ since $(h,\beta)$ is decreasing. However, again, we do not know \red{which} index $k$ \red{minimizes} $\Fu(k,h,\beta)$.
	}
		\begin{algorithm}[h!]
			\DontPrintSemicolon
			
			\Input{$u\in\Lambda$ with $\gsls{u}\neq \emptyset$ and \blue{$(h,\beta)\in\fcste(u)$ useful with $cst(h,\beta)=0$}} 
			\Output{$\valu{u}$ and $\bigkse_u=\min\agmx(u)$}
			\Begin{
				$k=0$; $K=+\infty$, $\vmax=-\infty$, $\kmax=0$\;
				\While{$k\leq K$}
				{
					\If{$k\in\res(u,h)$}{
						\If{$\vmax< u_k$}{
							$K=\lfloor\Fu(k,h,\beta)\rfloor$\;
							$\vmax=u_k$\;
							$\kmax=k$\;
						}
					}
					$k=k+1$\;
				}
				$\valu{u}=\vmax$\;
				$\bigkse_u=\kmax$\;
			}
			\caption{Resolution of Problem~\eqref{mainpb} with a useful constant upper bound}
			\label{algomaincst}
		\end{algorithm}
		In Algorithm~\ref{algomaincst}, when $(h,\beta)$ is chosen in $\fcste(u)$ and is such that $\cst(h,\beta)=0$, we can exploit the second statement of Corollary~\ref{mainpropmon} and update the truncation index $K$ only when $\vmax$ is updated. The third statement of \blue{Corollary~\ref{intercst}} proves that the stopping integer cannot be updated once $\bigkse_u$ \red{is reached}.
		
\blue{		
\begin{remark}
For simplicity, the algorithm based on an element $(h,\beta)\in\fcste(u)$ has not been presented. This algorithm would be a combination of the two algorithms that we provide. Within the "if" branch in Algorithm~\ref{algomaingen} on Line 10, we would maintain the same update for the truncation index until $c(h,\beta)-1$ (if the index lies in $\res(u,h)$). When $c(h,\beta)$ would be achieved, we would switch to the truncation index update on Line 6 in Algorithm~\ref{algomaincst}, still under the conditions $k\in \res(u,h)$ and $vmax<u_k$. 
\end{remark}		
}
		\section{Applications}
		\label{applis}

\blue{
We illustrate our techniques using concrete sequences, including the Fibonacci, logistic, and Syracuse sequences, as well as the (linear) norm-based peak computation problem.   
}		
		
		\subsection{A simple case}
		\red{First, we begin} with a very simple example to illustrate our approach. Let us consider, for $a\in\nn^*$, the sequence $x:=(a^n (n!)^{-1})_{n\in\nn}$. As presented in Problem~\ref{mainpb}, we are interested in the term $x_n$ which maximizes the sequence $x$ and \red{a} maximizer index. We apply our techniques to compute \red{these values}. It is evident that alternative approaches (the study of the variation of the sequence) can solve the problem \red{more easily}. For all $n\in\nn$, the terms $x_n$ \red{can be rewritten as}
		\[
		x_n=\dfrac{a^n}{n!}=\left(\dfrac{a}{a+1}\right)^n\dfrac{(a+1)^{n}}{n!}.
		\]
		Hence, for all $n\in\nn$, \red{we have}
		\[
		x_n=h_n(\beta^n)\text{ where } h_n:[0,1]\ni t\to t\dfrac{(a+1)^n}{n!} \text{ and }\beta=\dfrac{a}{a+1}
		\]
		It is straightforward to see that $h=(h_n)_{n\in\nn}$ is eventually decreasing and $\mon(h)=a$. Moreover, as for all $n\in\nn$, $h_n(0)=0$, we have $\res(x,h)=\nn$. Then, following Algorithm~\ref{algomaingen}, we compute $\mathfrak{F}_x(k,h,\{\beta\}_{n\in\nn})$ from $k=a$. Then we have:
		\[
		\mathfrak{F}_x(a,h,\{\beta\}_{n\in\nn})=\dfrac{\ln\left(h_a^{-1}(x_a)\right)}{\ln(\beta)}=\dfrac{\ln\left(\dfrac{a^a}{a!}\times \dfrac{a!}{(a+1)^a}\right)}{\ln\left(\dfrac{a}{a+1}\right)}=
		\dfrac{a\ln\left(\dfrac{a}{a+1}\right)}{\ln\left(\dfrac{a}{a+1}\right)}=a.
		\]
		We conclude that we have to compare the $a+1$ first terms of the sequence $x$ to find $\valu{x}$ and \red{an} integer maximizer of $x$. Using alternative approaches, we know that the maximal term of the sequence $x$ is \red{achieved} at $k=a$.
		
		This example shows why we allow sequences of functions in \red{the} upper bounds of the analyzed sequence, whereas Theorem~\ref{funatt} suggests \red{using} a constant sequence determined by an affine function. For example, \blue{by} defining $g:t\to t(a+1)^a$, we have for all $n\in\nn$, $x_n\leq g(\beta^n)$. However, as $h_n\leq g$ for all $n\in\nn$, we have from Corollary~\ref{coroinc} $\blue{\mathfrak{F}_x}(a,h,\{\beta\})\leq \blue{\mathfrak{F}_x}(a,\{g\},\{\beta\})$. The main advantage of using a constant sequence is that we can compute a truncation index from $n=0$ and $\blue{\mathfrak{F}_x}(n,\{g\},\{\beta\})$ is an upper bound for $\bigks_x$ for all $n\in\nn$. More precisely, for all $n\in\nn$, \red{we have}
		\[
		\mathfrak{F}_x(n,\{g\},\{\beta\})=\dfrac{\ln(g^{-1}(x_n)}{\ln(\beta)}=\dfrac{\ln\left(\dfrac{a^n}{a!(a+1)^a}\right)}{\ln\left(\dfrac{a}{a+1}\right)}=
		\dfrac{\ln\left(\dfrac{a^{n-a}}{a!}\right)}{\ln\left(\dfrac{a}{a+1}\right)}+a.
		\] 
		We \blue{observe} the third statement of Theorem~\ref{mainprop}, that is, as $x$ is increasing on $[0,a]\cap \nn$, $\mathfrak{F}_x(n,\{g\},\{\beta\})$ is decreasing on $[0,a]\cap \nn$. The smallest value of $\mathfrak{F}_x(n,\{g\},\{\beta\})$ is achieved at $n=a$ and
		\[
		\mathfrak{F}_x(a,\{g\},\{\beta\})=
		-\dfrac{\ln(a!)}{\ln\left(\dfrac{a}{a+1}\right)}+a.
		\]
		This leads to, for example, \blue{for $a=2$,} $\lfloor\mathfrak{F}_x(2,\{g\},\{\beta\})\rfloor=3$, \blue{for $a=5$,}  $\lfloor\mathfrak{F}_x(5,\{g\},\{\beta\})\rfloor=31$ and \blue{for $a=10$,} $\lfloor\mathfrak{F}_x(10,\{g\},\{\beta\})\rfloor=168$. The overapproximation gap completely explodes. This justifies \red{the selection of} the tightest possible upper bound \blue{for} the analyzed sequence, even if the upper bound is not a constant sequence.
		
		\subsection{Logistic \red{sequences}}
		Let \red{us} consider \red{the} logistic sequences parameterized by $r\in (0,4]$ starting at $y_0\in (0,1)$. More precisely, we have for all $n\in\nn$:
		\[
		y_{n+1}:=ry_n(1-y_n)
		\]
		We can exploit closed form \blue{expressions} for $r\in (0,1)$. For $r\in (0,1)$, from~\cite[Lemma A.6]{campbell2019automated}, we have at hand the upper-bound for all $n\in\nn$, $y_n\leq y_0(r^{-n}+y_0 n)^{-1}$. We then have for all $n\in\nn$, $y_n\leq h_n(\beta^n)$ where $h_n:(0,1]\ni t\mapsto y_0(t^{-1}+n y_0)^{-1}$ and $h_n(0)=0$ and $\beta=r$. Then, for all $n\in\nn$, $h_n$ belongs to $\fset$ and the sequence $(h_n)_n$ is decreasing. Moreover, $y_0=h_0(1)$ and we can then apply the third statement of Proposition~\ref{nonemptyfun} to conclude that $\valu{y}=y_0$ and $\agmx(y)=\{0\}$. 
		
		In the general setting $r\in (0,4]$, we can always find an upper bound $(h(\beta^n))_n$ for $y$. Indeed, as $t\mapsto rt(1-t)$ is bounded by $r/4$ on $[0,1]$. Then we have $y_n\leq r/4+\beta^n$ for all $\beta\in [0,1)$. Hence, we have $y_n\leq h(\beta^n)$ where $h:x\mapsto r/4+x$. However, this function $h$ is not useful since $y_n\leq h(0)$ for all $n\in\nn$. The key difficulty is to find a useful functional upper bound \red{for} the sequence. 
		
		\subsection{The sequence of the ratio of two consecutive terms of a Fibonacci sequence}
		In this example, we consider the ratio \blue{of two consecutive terms} of some Fibonacci sequences starting from $u_0\geq 0$ and $u_1>u_0\phi$ (for simplicity), where $\phi=(1+\sqrt{5})/2$ (the golden number). By the definition of Fibonacci sequences, we have for all $n\in\nn$, $u_{n+2}=u_{n+1}+u_n$. It is well-known that for all $n\in\nn$, $u_n=A\phi^n+B(-\phi)^{-n}$ where $A=\frac{u_0+u_1\phi}{1+\phi^2}$ and 
		$B=\frac{(u_0\phi-u_1)\phi}{1+\phi^2}$. From, the assumption on $u_0$ and $u_1$, we have $A>0$ and $B<0$ but $u_n>0$ for all $n\in\nn^*$. Let us define:
		\[
		\blue{z}_0:=\left\{\begin{array}{lr}u_{1}/u_0 & \text{ if } u_0> 0\\
			0 & \text{ if } u_0=0\end{array}
		\right. \text{ and for all } n\in\nn^*, \blue{z}_n:=\dfrac{u_{n+1}}{u_n} \enspace .
		\]
		We have for all $n\in\nn^*$
		\[
		\blue{z}_n=\phi\left(1+\red{(-1)^{n+1}\frac{B\phi^{-2}+B}{A\phi^{2n}+B(-1)^n}}\right)\enspace .
		\]
		This proves that $\ls{\blue{z}}=\lim_{n\to +\infty} \blue{z}_n=\phi$. This also proves that for all odd $n$, $\blue{z}_n<\phi$ and for all \red{non zero even} $n$, $\blue{z}_n>\phi$. Now, as $B<0$ \blue{and for all $k\in\nn^*$, $A\phi^{2k}+B>0$}, we have for all $n\in\nn^*$
		\[
		\blue{z}_n\leq h(\beta^n)\text{ where } h:x \to \left\{\begin{array}{lr}\phi\left(1-\dfrac{B\phi^{-2}+B}{Ax^{-1}+B}\right) & \text{ if } x\neq 0\\ \\ \phi & \text{ if } x=0\end{array}\right. \text{ and }\beta=\phi^{-2}
		\] 
		with equality when $n$ is even. 
		
		The function $h$ is strictly increasing and continuous on $[0,1]$ and $\beta\in (0,1)$  thus \blue{$(\{h\}_n,\{\beta\}_n)\in\fcste(\blue{z})$ with $\cst(h,\beta)=0$}. When $n\in\nn^*$ is even and thus the equality holds, we have $h^{-1}(\blue{z}_n)=(\phi^{-2})^n$. For the case where $n=0$, we also have $\blue{z}_0\leq h(1)$. If $u_0$ is null, \blue{then } $0$ does not belong to $\res(\blue{z},h)$. If $u_0$ is strictly positive, the equality $\blue{z}_0=h(1)$ holds, and we can conclude from the third statement of Proposition~\ref{nonemptyfun} that $\valu{\blue{z}}=\blue{z}_0$ and $\agmx(\blue{z})=\{0\}$.
		
		Now, suppose that $u_0=0$. We can \blue{compute} $\mathfrak{F}_\blue{z}(k,\{h\}_{n\in\nn},\{\beta\}_{n\in\nn})$ for all $k$ such that $\blue{z}_k>h(0)=\phi$. Actually, $\res(\blue{z},h)=\{2k:k\in\nn^*\}$.
		For all $k\in\nn^*$, \[\mathfrak{F}_\blue{z}(2k,\{h\}_{n\in\nn},\{\beta\}_{n\in\nn})=\dfrac{\ln(h^{-1}(\blue{z}_{2k}))}{\ln(\phi^{-2})}=(2k)\dfrac{\ln(\phi^{-2})}{\ln(\phi^{-2})}.\] 
		We conclude that $\bigks_\blue{z}\leq 2$ and since $\blue{z}_0$ and $\blue{z}_1$ are smaller than $\phi$, $\bigks_\blue{z}=2$. 
		
		\subsection{Syracuse sequences}
		
		Let \red{us} consider a (accelerated) Syracuse iteration starting at $N\in \nn^*$. Formally, we define the following sequence
		\begin{equation}
			\label{syra}
			\blue{\omega}_0:=N\text{ and for all }n\in\nn,\ 
			\blue{\omega}_{n+1}:=\left\{ 
			\begin{array}{lr}
				\dfrac{\blue{\omega}_n}{2}& \text{ if } \blue{\omega}_n \text{ is even}\\
				\\
				\dfrac{3\blue{\omega}_n+1}{2}& \text{ if } \blue{\omega}_n \text{ is odd}
			\end{array}\right.
		\end{equation}
		
		The Collatz conjecture asserts that for all $N\in\nn^*$, the sequence $\blue{\omega}=(\blue{\omega}_n)_n$ defined according to \eqref{syra} admits an index $k\in\nn$ such that $\blue{\omega}_k=1$. 
		
		\blue{If $\omega_0=N\in\{1,2\}$, then $\valu{y}$ and $\ls{y}$ are equal to 2. In this case, following the second statement of Proposition~\ref{argmax}, our method cannot be applied. However, it is obvious that, if $N\in\{1,2\}$, there exists $k\in\nn$ such that $\omega_k=1$. Thus, we can assume that $N\geq 3$ and the Collatz conjecture is understood for any starting integer $N$ greater than 3.}   
		
		
		\begin{theorem}
			The Collatz conjecture holds \blue{(for $N\geq 3$)} if and only if there exists $(a,b,c)\in \rr_+\times (0,1)\times \blue{(2,3]}$ such that $\blue{\omega}_n\leq ab^n +c$ for all $n\in\nn$.  
		\end{theorem}
		
		\begin{proof}
			If the Collatz conjecture holds, then \blue{$\ls{\blue{\omega}}=2$} as \blue{the sequence $\omega$ equals $(1,2,1,2,\ldots)$ from the index $k$ such that $\omega_k=1$.} Since we suppose that \blue{$N\geq 3$, $\ls{\blue{\omega}}<N\leq \valu{\blue{\omega}}$ and thus } $\gsls{u}$ is nonempty. We can follow the proof of \blue{Th.}~\ref{funatt} and choose $c\in \blue{(2,3]}\subset (\ls{\blue{\omega}},\valu{\blue{\omega}})$. Again, following the proof of Th~\ref{funatt}, we deduce suitable $a$ and $b$ to obtain the inequality.
			
			Suppose that the inequality holds with $(a,b,c)\in \rr_+\times (0,1)\times \blue{(2,3]}$. As $c\in \blue{(2,3]}$, then $\ls{\blue{\omega}}\leq c\leq \blue{3}$. As $y$ is a sequence of strictly positive integers, it implies that there exists $M\in\nn$ such that $\blue{\omega}_n\in\blue{\{1,2,3\}}$ for all $n\geq M$. The value \blue{3 is not possible}. Indeed, if $\blue{\omega}_n=3$, then $\blue{\omega}_{n+1}=5$ which is not less than 3. \blue{We conclude that  $\omega_n\in\{1,2\}$ for all $n\geq M$ which proves the result.} 
		\end{proof}
		
		\subsection{Linear \red{discrete}-time dynamical systems}
		
		\subsubsection{General case}
		Let \red{us} consider a $d\times d$ matrix $A$ and the trajectories $x_{k+1}= A x_k$ starting from $x_0\in\rd$. We suppose that $A$ is stable\red{, that is,} its spectral radius is strictly less than 1.
		
		Recall that all positive definite $d\times d$ matrices $Q$ define a norm on $\rd$ as follows: $\norm{x}_Q=\sqrt{x^\intercal Q x}$. From this, we can define an operator norm for $A$ associated \red{with} $Q$ as follows: $\norm{A}_Q^2=\max_{x\neq 0}  \norm{Ax}_Q/\norm{x}_Q$. This norm is submultiplicative i.e. $\norm{BC}_Q\leq \norm{B}_Q\norm{C}_Q$ and thus implies that $\norm{A^k}_Q\leq \norm{A}_Q^k$. The fact that $A$ is stable is equivalent to \red{the fact} that for all matrix norms $\norm{A^k}$ tends to 0 as $k$ goes to $+\infty$.  
		
		In \cite{ahiyevich2018upper,dowler2013bounding,7984147}, the authors are interested in the maximum of the Euclidean matrix norm of the powers of $A$ i.e. $\max_{k\in\nn}\|A^k\|_2$ recalling that $\norm{A}_2^2=\max_{x\neq 0}\norm{Ax}_2/\norm{x}_2$. Following~\cite{adje13052025}, we can use \red{the} solutions of the discrete Lyapunov equation\red{, that is,} a matrix from $\mathcal{L}_A:=\{P \succ 0 : P-A^\intercal P A\succ 0\}$ where the symbol $R \succ 0$ means that $R$ is positive definite. Now, taking $P\in\mathcal{L}_A$, we can define the operator norm $\norm{A}_P$ as $P$ is positive definite. The main advantage of using \red{the} solutions of the discrete Lyapunov equation is to ensure that $\norm{A}_P<1$. Moreover, let $k\in\nn$ and $x\in\rd$ such that $A^k x\neq 0$. Then:
		
		\renewcommand{\arraystretch}{2}
		
		\[
		\begin{array}{lcl}
			\displaystyle{\dfrac{x^\intercal {A^k}^\intercal A^k x}{x^\intercal x}}&=&\displaystyle{\dfrac{x^\intercal {A^k}^\intercal A^k x}{x^\intercal {A^k}^\intercal P {A^k} x}\dfrac{x^\intercal {A^k}^\intercal P A^k x}{x^\intercal P x}\dfrac{x^\intercal P x}{x^\intercal x}}\\
			&\leq &\displaystyle{\max_{y\neq 0} \dfrac{y^\intercal y}{y^\intercal P y}\norm{A^k}_P^2\lmax{P}}\\
			&=&\displaystyle{\dfrac{1}{\displaystyle{\min_{y\neq 0} \frac{y^\intercal P y}{y^\intercal y}}}\norm{A^k}_P\lmax{P}}\\
			&=&\displaystyle{\dfrac{\lmax{P}}{\lmin{P}}\norm{A^k}_P^2}\leq\displaystyle{\dfrac{\lmax{P}}{\lmin{P}}\norm{A}_P^{2k}}
		\end{array}
		\]
		If $A^k$ is not invertible, the inequality also holds for all non zero $x\in\rd$ in the kernel of $A^k$. Finally, taking the maximum over all non zero $x\in\rd$, we \blue{obtain} for all $k\in\nn$
		\[
		\norm{A^k}_2^2\leq h_P(\beta_P^k) \text{ where } h_P:t \to \dfrac{\lmax{P}}{\lmin{P}}t \text{ and }
		\beta_P=\norm{A}_P^2\enspace.
		\]
		Finally, let for all $k\in\nn$, $\blue{\nu}_k:=\norm{A^k}_2^2$ and define $\blue{\nu}=(\blue{\nu}_k)_{k\in\nn}$ we have $(h_P,\beta_P)\in\fcste(\blue{\nu})$ with $\cst(h_P,\beta_P)=0$. We can apply our method to compute the maximal value of \red{the} sequence $\blue{\nu}$. As the terms $\blue{\nu}_k$ are strictly positive and for all $P\in\mathcal{L}_A$, $h_P(0)=0$, we have $\res(\blue{\nu},h_P)=\nn$. Moreover, we have
		\[
		\mathfrak{F}_{\blue{\nu}}(k,\{h_P\}_{n\in\nn},\{\beta_P\}_{n\in\nn})
		=\dfrac{\ln\left(h_P^{-1}(\blue{\nu}_k)\right)}{\ln(\beta_P)}=\dfrac{\ln\left(\dfrac{\lmin{P}}{\lmax{P}}\norm{A^k}_2^2\right)}{2\ln(\norm{A}_P}.
		\]
		which provides \red{upper bounds for} the greatest maximizer of $\blue{\nu}$. 
		
		\renewcommand{\arraystretch}{1}
		
		\subsubsection{Numerical example}
		Let us consider \red{an} example developed in~\cite{ahiyevich2018upper}. Let us define $U$ as the $d\times d$ matrix such that $U_{1d}=1$ and $U_{ij}=0$ otherwise, and $\Idd$ \red{as} the identity matrix (of size $d\times d$). Now, let $\lambda$ be in $(-1,1)$ and define the matrix \blue{as follows}:
		\[
		A_\lambda:=\begin{pmatrix}
			\lambda & 0 & \ldots & 1\\
			0 & \ddots&\ddots & \vdots \\
			\vdots& \vdots& \lambda& 0\\ 
			0  &\ldots& 0& \lambda 
		\end{pmatrix}=\lambda \Idd + U
		\]
		Then, as $U^k=0$ for all $k\geq 2$, we have for all $k\in\nn$ $A_\lambda^k=(\lambda Id+U)^k=\lambda^k\Idd+k \lambda^{k-1}U$. Moreover, $\norm{A_\lambda^k}_2^2=\lmax{(A_\lambda^k)^\intercal A_\lambda^k}=\lambda^{2k-2}(\lambda^2+0.5k^2+0.5k\sqrt{4\lambda^2+k^2})$. 
		Thus, the goal is to compute $\max_{k\in\nn} \norm{A_\lambda^k}_2^2$ by giving an upper-bound of the (greatest) maximizer of $\max_{k\in\nn} \norm{A_\lambda^k}_2^2$.
		
		\red{The set} $\mathcal{L}_{A_\lambda}$ contains particular diagonal matrices. Indeed,
		$P\in \mathcal{L}_{A_\lambda}$ if and only if $P\succ 0$ and :
		\[
		P-(A_\lambda)^\intercal P (A_\lambda)=P(1-\lambda^2) -\lambda (U^\intercal P+ PU)+U^\intercal P U \succ 0
		\]
		To simplify, we consider diagonal matrices $P$, which leads to \red{considering} positive definite diagonal matrices $P$ such that the matrix
		\[
		P-A_\lambda^\intercal P A_\lambda=\begin{pmatrix}
			(1-\lambda^2)p_{1,1} & 0 & \ldots & -\lambda p_{11}\\
			0& (1-\lambda^2)p_{2,2} & \ddots & \vdots\\
			\vdots & \vdots &  \ddots & 0\\
			-\lambda p_{11} & 0 &\ldots & (1-\lambda^2) p_{d,d}-p_{1,1}
		\end{pmatrix}
		\]
		is positive definite.
		
		Then, we can take $p_{1,1}=p_{2,2}=\cdots=p_{d-1,d-1}=1$ and in this case, $P-A_\lambda^\intercal P A_\lambda \succ 0$ if and only if its determinant is strictly positive. As $\det(P-A_\lambda^\intercal P A_\lambda)=(1-\lambda^2)^{d-2}(p_{d,d}(1-\lambda^2)^2-1)$, we conclude that $p_{d,d}$ must be strictly greater than $(1-\lambda^2)^{-2}$. Finally, we define \[P_q:=\operatorname{Diag}(1,\cdots,1,q)\] where $\operatorname{Diag}(v)$ is the diagonal matrix for which the diagonal is $v$ and $q\in \left((1-\lambda^2)^{-2},+\infty\right)$. For \red{the} matrices $P_q$, we have
		\[
		\norm{A_\lambda}_{P_q}^2=\dfrac{1+2q\lambda^2+\sqrt{1+4\lambda^2 q}}{2q}
		\]  
		Finally,
		\[
		\mathfrak{F}_{\blue{\nu}}(k,\{h_{P_q}\}_{n\in\nn},\{\beta_{P_q}\}_{n\in\nn})
		=\dfrac{\ln\left(\dfrac{\lambda^{2k-2}(\lambda^2+0.5k^2+0.5k\sqrt{4\lambda^2+k^2}}{q}\right)}{\ln\left(\dfrac{1+2q\lambda^2+\sqrt{1+4\lambda^2 q}}{2q}\right)}.
		\]
		For our experiments, we \blue{select} $q=2(1-\lambda^2)^{-2}$. Table~\ref{smallbenchs} \blue{lists} the computations for several values of $\lambda$.
		\begin{center}
			\begin{table}[h!]
				\begin{center}
					\begin{tabular}{|c||c|c|c|}
						\hline
						& $\bigks_\blue{\nu}$ & $\displaystyle{\max_{k\in\nn} \blue{\nu}_k}$ & $\mathfrak{F}_{\blue{\nu}}(\bigks_{\blue{\nu}})$ \\
						\hline
						$\lambda=0.1$& 0 &1 &1\\
						\hline
						$\lambda=0.25$&0 &1 &1\\
						\hline
						$\lambda=0.5$& 1 & 1.4572 & 2\\
						\hline
						$\lambda=0.75$& 3 &3.1937 &7\\
						\hline
						$\lambda=0.9$&9 &15.3082 &20\\
						\hline
						$\lambda=0.99$&99 &1367.27 &221\\
						\hline
						$\lambda=0.999$&999 &135471  &2229\\
						\hline
						$\lambda=0.99995$&19999 &54136820.5  &44617\\
						\hline
					\end{tabular}
				\end{center}
				\caption{Numerical results}
				\label{smallbenchs}
			\end{table}
		\end{center}
		\subsubsection{Going beyond}
		As suggested in this paper, we could consider non constant sequences $(P_k)_{k\in\nn}$ in $\mathcal{L}_{A_\lambda}$. Following Corollary~\ref{maincoro}, we must impose \red{a} decreasing condition, that is, in this case, for all $k\in\nn$ 
		\[
		\dfrac{\lmax{P_{k+1}}}{\lmin{P_{k+1}}}\leq \dfrac{\lmax{P_{k}}}{\lmin{P_{k}}}\text{ and } \norm{A_\lambda}_{P_{k+1}}\leq \norm{A_\lambda}_{P_{k}}
		\]
		Note that $\norm{A_\lambda}_P$ is lower bounded by the spectral radius of $A_\lambda$, which is equal to $|\lambda|$. This was also suggested in~\cite{adje13052025} by considering a minimization problem over $\mathcal{L}_{A_\lambda}$
		\[
		\begin{array}{rcl}
			\operatorname{Min} & & F_k(P) \\  
			& \text{s.t.} & P\in\mathcal{L}_{A_\lambda}
		\end{array}
		\text{ where } F_k(P):=\dfrac{\ln\left(\dfrac{\lmin{P}}{\lmax{P}}\norm{A^k}_2^2\right)}{2\ln(\norm{A_\lambda}_P}
		\]
		In this case, monotonicity is a stronger condition  (from the integer $N$, if any, such that the sequence $(\norm{A_\lambda^k}_2)_k$ is decreasing ) \red{than} the condition $P_k$ minimizes $F_k$ over $\mathcal{L}_{A_\lambda}$.   
		
		\bibliographystyle{plain}
		\bibliography{shortnotebibfinal.bib}

\begin{thebibliography}{10}

\bibitem{DBLP:journals/jota/Adje21}
A.~Adj{\'{e}}.
\newblock {Quadratic Maximization of Reachable Values of Affine Systems with
  Diagonalizable Matrix}.
\newblock {\em J. Optim. Theory Appl.}, 189(1):136--163, 2021.

\bibitem{adje2015property}
A.~Adj{\'e}, P.-L. Garoche, and V.~Magron.
\newblock {Property-based Polynomial Invariant Generation using Sums-of-Squares
  Optimization}.
\newblock In {\em Static Analysis: 22nd International Symposium, SAS 2015,
  Saint-Malo, France, September 9-11, 2015, Proceedings 22}, pages 235--251.
  Springer, 2015.

\bibitem{adje13052025}
A.~Adjé.
\newblock Quadratic {M}aximization of {R}eachable {V}alues of {S}table
  {D}iscrete-{T}ime {A}ffine {S}ystems.
\newblock {\em To appear in Optimization}, 0(0):1--52, 2025.

\bibitem{ahiyevich2018upper}
U.~M. Ahiyevich, S.~E. Parsegov, and P.~S. Shcherbakov.
\newblock {Upper Bounds on Peaks in Discrete-Time Linear Systems}.
\newblock {\em Automation and Remote Control}, 79:1976--1988, 2018.

\bibitem{ahmadi2025robust}
Amir~Ali Ahmadi and Oktay G{\"u}nl{\"u}k.
\newblock Robust-to-{D}ynamics {O}ptimization.
\newblock {\em Mathematics of Operations Research}, 50(2):965--992, 2025.

\bibitem{Bestuzheva2025}
K.~Bestuzheva, A.~Chmiela, B.~M{\"u}ller, F.~Serrano, S.~Vigerske, and
  F.~Wegscheider.
\newblock Global {O}ptimization of {M}ixed-{I}nteger {N}onlinear {P}rograms
  with {SCIP 8}.
\newblock {\em Journal of Global Optimization}, 91(2):287--310, Feb 2025.

\bibitem{campbell2019automated}
Trevor Campbell and Tamara Broderick.
\newblock Automated {S}calable {B}ayesian {I}nference via {H}ilbert {C}oresets.
\newblock {\em Journal of Machine Learning Research}, 20(15):1--38, 2019.

\bibitem{10.11650/twjm/1500407524}
C.~C. Cowen and B.~D. MacCluer.
\newblock {S}chroeder’s {E}quation in {S}everal {V}ariables.
\newblock {\em Taiwanese Journal of Mathematics}, 7(1):129 -- 154, 2003.

\bibitem{MR1058437}
B.~A. Davey and H.~A. Priestley.
\newblock {\em Introduction to {L}attices and {O}rder}.
\newblock Cambridge Mathematical Textbooks. Cambridge University Press,
  Cambridge, 1990.

\bibitem{dowler2013bounding}
Daniel~Ammon Dowler.
\newblock {\em Bounding the {N}orm of {M}atrix {P}owers}.
\newblock Brigham Young University, 2013.

\bibitem{HASSIN2018795}
Refael Hassin and Anna Sarid.
\newblock Operations {R}esearch {A}pplications of {D}ichotomous {S}earch.
\newblock {\em European Journal of Operational Research}, 265(3):795--812,
  2018.

\bibitem{Hemmecke2010}
R.~Hemmecke, M.~K{\"o}ppe, J.~Lee, and R.~Weismantel.
\newblock {\em Nonlinear {I}nteger {P}rogramming}, pages 561--618.
\newblock Springer Berlin Heidelberg, Berlin, Heidelberg, 2010.

\bibitem{Kreiss1962}
H.-O. Kreiss.
\newblock {\"U}ber die {S}tabilit{\"a}tsdefinition {F}{\"u}r
  {D}ifferenzengleichungen die {P}artielle {D}ifferentialgleichungen
  {A}pproximieren.
\newblock {\em BIT Numerical Mathematics}, 2(3):153--181, Sep 1962.

\bibitem{lagarias2023ultimate}
Jeffrey~C Lagarias, editor.
\newblock {\em The {U}ltimate {C}hallenge: The $3 x+ 1$ {P}roblem}.
\newblock American Mathematical Society, 2023.

\bibitem{li2006nonlinear}
D.~Li and X.~Sun.
\newblock {\em Nonlinear {I}nteger {P}rogramming}.
\newblock Springer, 2006.

\bibitem{miller2020peaksafety}
J.~Miller, D.~Henrion, and M.~Sznaier.
\newblock {Peak Estimation Recovery and Safety Analysis}.
\newblock {\em IEEE Control Systems Letters}, 5(6):1982--1987, 2020.

\bibitem{miller2021peakdiscrete}
Jared Miller, Didier Henrion, Mario Sznaier, and Milan Korda.
\newblock {Peak Estimation for Uncertain and Switched Systems}.
\newblock In {\em 2021 60th IEEE Conference on Decision and Control (CDC)},
  pages 3222--3228. IEEE, 2021.

\bibitem{misener2014antigone}
R.~Misener and C.~A. Floudas.
\newblock {ANTIGONE}: {A}lgorithms for {C}ontinuous/{I}nteger {G}lobal
  {O}ptimization of {N}onlinear {E}quations.
\newblock {\em Journal of Global Optimization}, 59(2):503--526, 2014.

\bibitem{murota2009recent}
Kazuo Murota.
\newblock Recent {D}evelopments in {D}iscrete {C}onvex {A}nalysis.
\newblock {\em Research Trends in Combinatorial Optimization: Bonn 2008}, pages
  219--260, 2009.

\bibitem{jugg}
C.A. Pickover.
\newblock {\em Computers and the {I}magination: {V}isual {A}dventures {B}eyond
  the {E}dge}.
\newblock St. Martin's Press, 1992.

\bibitem{polyak2018peak}
Boris~T Polyak, Pavel~S Shcherbakov, and Georgi Smirnov.
\newblock Peak {E}ffects in {S}table {L}inear {D}ifference {E}quations.
\newblock {\em Journal of Difference Equations and Applications},
  24(9):1488--1502, 2018.

\bibitem{prasadestimates}
V~Prasad and MA~Prasad.
\newblock Estimates of the {M}aximum {E}xcursion {C}onstant and {S}topping
  {C}onstant of {J}uggler-like {S}equences.
\newblock {\em preprint}, 2025.

\bibitem{MR2446182}
Steven Roman.
\newblock {\em Lattices and {O}rdered {S}ets}.
\newblock Springer, New York, 2008.

\bibitem{7984147}
P.~Shcherbakov.
\newblock {On Peak Effects in Discrete Time Linear Systems}.
\newblock In {\em 2017 25th Mediterranean Conference on Control and Automation
  (MED)}, pages 376--381, 2017.

\bibitem{silva1999maximum}
Tom{\'a}s Oliveira~E Silva et~al.
\newblock Maximum {E}xcursion and {S}topping {T}ime {R}ecord-{H}olders for the
  $3x+ 1$ {P}roblem: {C}omputational {R}esults.
\newblock {\em Mathematics of Computation}, 68(225):371--384, 1999.

\bibitem{Spijker1991}
M.~N. Spijker.
\newblock On a {C}onjecture by le {V}eque and {T}refethen {R}elated to the
  {K}reiss {M}atrix {T}heorem.
\newblock {\em BIT Numerical Mathematics}, 31(3):551--555, Sep 1991.

\bibitem{tao2022almost}
T.~Tao.
\newblock Almost all {O}rbits of the {C}ollatz {M}ap {A}ttain {A}lmost
  {B}ounded {V}alues.
\newblock {\em Forum Math. Pi}, 10:Paper No. e12, 56, 2022.

\bibitem{Terras1976}
Riho Terras.
\newblock A {S}topping {T}ime {P}roblem on the {P}ositive {I}ntegers.
\newblock {\em Acta Arithmetica}, 30(3):241--252, 1976.

\bibitem{TrefethenEmbree+2020}
L.~N. Trefethen and M.~Embree.
\newblock {\em Spectra and {P}seudospectra: the {B}ehavior of {N}onnormal
  {M}atrices and {O}perators}.
\newblock Princeton University Press, Princeton, 2020.

\bibitem{4303250}
James~F. Whidborne and John McKernan.
\newblock On the {M}inimization of {M}aximum {T}ransient {E}nergy {G}rowth.
\newblock {\em IEEE Transactions on Automatic Control}, 52(9):1762--1767, 2007.

\bibitem{ZHANG2023205}
Jiayi Zhang, Chang Liu, Xijun Li, Hui-Ling Zhen, Mingxuan Yuan, Yawen Li, and
  Junchi Yan.
\newblock A {S}urvey for {S}olving {M}ixed {I}nteger {P}rogramming via
  {M}achine {L}earning.
\newblock {\em Neurocomputing}, 519:205--217, 2023.

\end{thebibliography}

	\end{document}